\documentclass{amsart}
\usepackage{graphicx}
\usepackage{amssymb}
\usepackage{epstopdf}
\usepackage{nicefrac}
\usepackage{enumerate}
\usepackage{verbatim}
\usepackage{hyperref}
\usepackage{color}
\usepackage{tikz}
\usepackage[normalem]{ulem}

\newtheorem{thm}{THEOREM}[section]

\newtheorem{cor}[thm]{COROLLARY}

\newtheorem{defn}[thm]{DEFINITION}
\newtheorem{ex}[thm]{EXAMPLE}

\newtheorem{lemma}[thm]{LEMMA}
\newtheorem{prob}[thm]{PROBLEM}
\newtheorem{prop}[thm]{PROPOSITION}

\newtheorem{remark}[thm]{REMARK}

\newcommand{\mN}{{\mathbb N}}

\newcommand{\mZ}{{\mathbb Z}}

\newcommand{\cG}{{\mathcal G}}

\newcommand{\cK}{{\mathcal K}}

\newcommand{\cS}{{\mathcal S}}

\newcommand{\cU}{{\mathcal U}}

\newcommand{\e}{{\epsilon}}

\input xy
\xyoption {all}

\begin{document}

\title{Measures and stabilizers of group Cantor actions}

\begin{abstract}
{ 
We consider a minimal equicontinuous action of a finitely generated group $G$ on a Cantor set $X$ with invariant probability measure $\mu$, and stabilizers of points for such an action. We give sufficient conditions under which there exists a subgroup $H$ of $G$ such that the set of points in $X$ whose stabilizers are conjugate to $H$ has full measure. The conditions are that the action is locally quasi-analytic and locally non-degenerate. An action is locally quasi-analytic if its elements have unique extensions on subsets of uniform diameter. The condition that the action is locally non-degenerate is introduced in this paper. We apply our results to study the properties of invariant random subgroups induced by minimal equicontinuous actions on Cantor sets and to certain almost one-to-one extensions of equicontinuous actions. 
}
\end{abstract}

\author{Maik Gr\"oger}
\address{Maik Gr\"oger, Faculty of Mathematics, University of Vienna, Oskar-Morgenstern-Platz 1, 1090 Vienna, Austria \& Faculty of Mathematics and Computer Science, Jagiellonian University in Krakow, ul. \L{}ojasiewicza 6, 30‐348 Krak\'ow, Poland}
\email{maik.groeger@univie.ac.at}

\author{Olga Lukina}
\address{Olga Lukina, Faculty of Mathematics, University of Vienna, Oskar-Morgenstern-Platz 1, 1090 Vienna, Austria}
 \email{olga.lukina@univie.ac.at}


\thanks{Affiliation: MG \& OL: Faculty of Mathematics, University of Vienna, Oskar-Morgenstern-Platz 1, 1090 Vienna, Austria; MG: Faculty of Mathematics and Computer Science, Jagiellonian University in Krakow, ul. \L{}ojasiewicza 6, 30‐348 Krak\'ow, Poland}


\thanks{MG is supported by DFG grant GR 4899/1-1, OL is supported by FWF Project P31950-N35}

\thanks{Version date: August 13, 2020}

\date{}


\maketitle



\section{Introduction}\label{sec-intro}

Let $G$ be a finitely generated group, and suppose $G$ acts on a Cantor set $X$ preserving a probability measure $\mu$. We denote the action by $(X,G,\Phi,\mu)$, where the assignment $\Phi: G \to Homeo(X)$ is a homomorphism, and use the short notation $g \cdot x = \Phi(g)(x)$ for the action of $g \in G$ on $x \in X$ throughout the paper. Throughout the paper, a \emph{clopen} set is a closed and open set in $X$.

Let $G_x = \{ g \in G \mid g \cdot x = x\}$ be the \emph{stabilizer}, or the \emph{isotropy subgroup} of the action of $G$ at $x$. A stabilizer $G_x$ is trivial if $G_x = \{e\}$, where $e$ is the identity in $G$. The action $(X,G,\Phi, \mu)$ is \emph{essentially free} if the subset of points with trivial stabilizers has full measure in $G$. If $y$ is in the orbit of $x$, then their stabilizers are conjugate subgroups, that is, $G_y = g \, G_x \, g^{-1}$ for some $g \in G$ such that $y = g \cdot x$. The following problem is natural.

\begin{prob}\label{prob-111}
 Let $(X,G,\Phi,\mu)$ be a group action with an invariant ergodic probability measure $\mu$. Formulate conditions on the action, under which there exists a subgroup $H \subset G$ such that stabilizers of points in a full measure subset of $X$ are conjugate to $H$.
\end{prob}

For example, if an action is essentially free, then for  $H$   the trivial group,   there is a full measure subset of $X$ of points with   trivial stabilizers. On the other hand, there exist many actions where no point has trivial stabilizer, and non-trivial stabilizers differ for different points. Vershik \cite{Vershik2011,Vershik2012} introduced the notion of an \emph{extremely non-free action}, for which the subset of points with pairwise distinct stabilizers has full measure in $X$. Many extremely non-free actions arise as self-similar actions of subgroups of the automorphism group $Aut(T)$ of a $d$-ary rooted tree $T$. Here the Cantor set $X$ is the boundary $\partial T$ of the tree $T$. Every automorphism $g \in Aut(T)$ corresponds to a homeomorphism in $Homeo(\partial T)$, and the measure $\mu_{\partial T}$ on $\partial T$ descends from the Haar measure on the closure of the action in $Homeo(\partial T)$. Every such action on $\partial T$ is equicontinuous with respect to the standard metric, see Section \ref{subsec-equicts} for definitions and details.

Any minimal equicontinuous action of a finitely generated group $G$ on a Cantor set $X$ is conjugate to an action of a subgroup of $Aut(T)$ on $\partial T$, for an appropriate choice of a tree $T$. The tree $T$ need not be $d$-ary and the action of $G$ need not be self-similar. Recent work of the second author has been devoted to developing invariants, quantifying the complexity of minimal equicontinuous actions and classifying them up to local conjugacy \cite{DHL2016,DHL2017,HL2019-2,HL2019}. 

In this paper, we give an answer to Problem \ref{prob-111} in the case when $(X,G,\Phi,\mu)$ is minimal and equicontinuous. The complexity of the action $(X,G,\Phi,\mu)$ is quantified using the notion of a \emph{locally quasi-analytic} action, introduced in \cite{ALC2009}, and studied for actions on Cantor sets in \cite{DHL2016,DHL2017,HL2019-2,HL2019}. The notion of a \emph{locally non-degenerate} action is introduced, which is key to our results, and examples of actions which are locally non-degenerate or locally degenerate are given. Known results about actions of self-similar groups on $d$-ary trees are obtained as consequences of our theorems. We apply our results to study properties of invariant measures on the space ${\rm Sub}(G)$ of closed subgroups of $G$, and certain almost one-to-one extensions of equicontinuous group actions.

\medskip
Unless specifically mentioned otherwise, $X$ is a Cantor set, $G$ is a finitely generated group, and $\mu$ is the invariant ergodic probability measure for a minimal equicontinuous action $(X,G,\Phi,\mu)$. The action $(X,G,\Phi,\mu)$ is uniquely ergodic. Recall that unique ergodicity is a topological property of $(X,G,\Phi,\mu)$, that is, it is invariant under topological conjugacy.

We now give precise statements of our results. When we refer to purely topological properties of the action, which do not require a measure, we omit the measure $\mu$ from the notation. 

An action $(X,G,\Phi)$ is \emph{topologically free} if the set of points with non-trivial stabilizer is meager in $X$. Since a meager set has empty interior, for a topologically free action the set of points with trivial stabilizer is dense in $X$. If an action is  minimal and essentially free, then it is topologically free. The converse need not hold, and there are minimal topologically free actions which are not essentially free \cite{AE2007,BG2004}. 

For actions where every point has non-trivial stabilizer, we can generalize the concept of an essentially free action by distinguishing \emph{points with trivial or non-trivial holonomy}, as we explain now.

Let $g \in G$, and let $x \in X$ such that $g \cdot x = \Phi(g)(x) = x$. The homeomorphism $\Phi(g)$ either fixes every point in an open neighborhood of $x \in X$, or it induces a non-trivial map on any open neighborhood of $x \in X$. We define the \emph{neighborhood stabilizer} of $x \in X$ by
 \begin{align}\label{eq-nbhdstab}[G]_x = \{g \in G_x \mid g|V_g = id \textrm{ for some open }V_g \owns x\}.\end{align}
The term neighborhood stabilizer comes from  \cite{Vor2012}. Note that the  open  set $V_g$ in \eqref{eq-nbhdstab} may depend on $g \in G$.  In Definition \ref{defn-holpoints} below, $X$ can be any metric space equipped with an action, and the action need not be either minimal or equicontinuous or measure-preserving.

\begin{defn}\label{defn-holpoints} Let $X$ be a metric space and  let $(X,G,\Phi)$ be an action.
\begin{enumerate}
\item We say that $g \in G$ \emph{has trivial holonomy} at $x \in X$ if $g \in [G]_x$, and otherwise $g$ \emph{has non-trivial holonomy} at $x$.
\item We say that $x \in X$ is a \emph{point with trivial holonomy}, or a \emph{point without holonomy}, if $G_x = [G]_x$. Otherwise $x \in X$ is a \emph{point with non-trivial holonomy}, or a \emph{point with holonomy}.
\end{enumerate}
\end{defn}

The terminology in Definition \ref{defn-holpoints} is standard in foliation theory, where it was developed in a more general context of actions of pseudogroups of local homeomorphisms on compact spaces. We recall a basic result by Epstein, Millett and Tischler \cite{EMT1977} for foliated spaces, which reads for the case of actions of groups on compact spaces as follows.

\begin{thm}\label{thm-residualwoholonomy}
Let $(X,G,\Phi)$ be an action of a countable group on a compact Hausdorff space. The set of points without holonomy is a residual (dense $G_\delta$) subset of $X$.
\end{thm}

If $(X,G,\Phi)$ is topologically free, then points in $X$ with trivial stabilizers have trivial holonomy. Conversely, suppose $g \in G$ such that $g \cdot x = x$. Since the set of points with non-trivial stabilizers has empty interior, then there is no open set $U \owns x$ such that $g$ fixes every point in $U$. Thus for a topologically free action, a point $x \in X$ has non-trivial stabilizer if and only if it has non-trivial holonomy. Then a measure-preserving topologically free action $(X,G,\Phi,\mu)$ is essentially free if and only if the set of points with non-trivial holonomy in $X$ has measure zero. 

An analogue of an essentially free action in the case when no point in $(X,G,\Phi,\mu)$ has trivial stabilizer is an action where the set of points with trivial holonomy has full measure in $X$.

\medskip
 The concept of a \emph{quasi-analytic action} was first introduced by Haefliger \cite{Haefliger1985} for actions on connected spaces, and it was generalized to include totally disconnected spaces by \'Alvarez L\'opez and Candel \cite{ALC2009}. We give a local version of this definition as formulated in \cite{DHL2017}. Actions with or without this property were studied in \cite{DHL2016,DHL2017,HL2018,HL2019,HL2019-2}. 

  \begin{defn} \cite[Definition~9.4]{ALC2009} \label{def-LQA} Let $X$ be a Cantor set with a metric $D$, and let $G$ be a countable group. An action $(X,G,\Phi)$  is   \emph{locally quasi-analytic}, or simply  \emph{LQA}, if there exists $\epsilon > 0$ such that for any open set $U \subset X$ with ${\rm diam} (U) < \e$, any open subset $V \subset U $ and any elements $g_1 , g_2 \in G$
 \begin{equation}\label{eq-lqaa}
  \text{if } ~~ \Phi(g_1)|V = \Phi(g_2)|V, ~ \text{ then}~~ \Phi(g_1)|U = \Phi(g_2)|U. 
\end{equation}
The action $(X,G,\Phi)$ is \emph{quasi-analytic} if one can choose $U = X$.
\end{defn}

Thus an action is LQA if and only if for all $g \in G$ the homeomorphisms $\Phi(g)$ have unique extensions on subsets of uniform diameter in $X$. 
Topologically free actions are quasi-analytic \cite[Proposition 2.2]{HL2018}. Examples of actions which are locally quasi-analytic but not quasi-analytic are easy to construct: take any quasi-analytic action, and add a finite number of elements which fix an open subset of $X$ but act non-trivially on its complement, as in \cite[Example A.4]{HL2018}. A minimal equicontinuous action of a finitely generated  nilpotent group is always locally quasi-analytic, in fact, a similar property holds for a more general class of groups with Noetherian property  \cite[Theorem 1.6]{HL2018}. Examples of actions which are not LQA include the actions of finite index torsion-free subgroups of $SL(n,\mZ)$, for $n \geq 3$, in \cite{HL2019}. 

\medskip
The idea to use Lebesgue density to study the measure of points with non-trivial holonomy was introduced by Hurder and Katok \cite{HK1987}, for linear holonomy in $C^1$ foliations of a smooth manifold $M$. In the setting of \cite{HK1987}, the differentiability of a foliation implies that for each holonomy linear map $L$, the Lebesgue density of the set of all fixed points of $L$ at fixed points with non-trivial holonomy is bounded away from $1$. Then the Lebesgue density theorem implies that the set of points with non-trivial linear holonomy must have zero measure in the transverse Euclidean section to the foliation. 

In absence of differentiability for actions on Cantor sets, we introduce the notion of a \emph{locally non-degenerate} action. It will follow that if an action is locally non-degenerate, then for any $g \in G$ the Lebesgue density of the set of fixed points of $g$ at points with non-trivial holonomy is bounded away from $1$. This condition on the action is combinatorial and it is defined using a representation of $(X,G,\Phi,\mu)$ onto the boundary of a spherically homogeneous tree $T$, which can be considered as a choice of coordinates for $(X,G,\Phi,\mu)$. We discuss such representations in more detail below and in Section \ref{subsec-tree}. We will show that the property of local non-degeneracy is invariant under conjugacy of actions, and so it is an invariant of an equivalence class of representations, associated to $(X,G,\Phi,\mu)$. 

Local non-degeneracy of an action implies the property of Lebesgue densities, discussed above, but we do not know at the moment if the converse statement is true. Since the precise relationship between Lebesgue densities and the notion of local non-degeneracy  is not clear at the moment, we are not able to give an intrinsic, coordinate-free definition of a locally non-degenerate action. A similar state of affairs is not uncommon in the study of foliated spaces, where often the only way to study an object is in terms of equivalence classes of coordinate charts representing it. We discuss this more below.

We introduce the concept of a \emph{locally non-degenerate action} for an action of a subgroup $G \subset Aut(T)$ on a spherically homogeneous tree $T$ first, see Section \ref{subsec-tree} for more details about actions on trees.

Let $V = \bigcup_{\ell \geq 0} V_\ell$ be the vertex set of a tree $T$, where $V_\ell$ is the $\ell$-th \emph{level} of $T$. A tree $T$ is \emph{rooted} if $V_0$ is a singleton. A tree $T$ is \emph{spherically homogeneous} if for every $\ell \geq 1$ there is a positive integer $n_\ell \geq 2$ such that every vertex in $V_{\ell-1}$ is joined by edges to $n_\ell$ vertices in $V_\ell$. The sequence $n = (n_1,n_2,\ldots)$ is called the \emph{spherical index} of $T$.  A spherically homogeneous tree $T$ is $d$-ary if there is a positive integer $d \geq 2$ such that $n_\ell = d$ for all $\ell \geq 1$. 

The boundary $\partial T$ is the set of all connected paths in $T$. With product topology, it is a Cantor set. For $v \in V$, we denote by $T_{v}$ the subtree through the vertex $v$, and by $\partial T_{v}$ the boundary of $T_{v}$. Then $\partial T_{v}$ is identified with a clopen subset of $\partial T$. Denote by $\mu_{\partial T}$ the counting measure on $\partial T$. An automorphism $g \in Aut(T)$ naturally induces a homeomorphism of $\partial T$, thus defining the homomorphism $\Phi$. So if $G \subset Aut(T)$, then we write $(\partial T, G, \mu_{\partial T})$ for the action on the tree, omitting $\Phi$ from the notation.

\begin{defn}\label{defn-uniformnonconstanti}
Let $T$ be a spherically homogeneous tree, and let $G \subset Aut(T)$. The action of $g \in G$ on $\partial T$ is \emph{locally non-degenerate} if there exists $0 < \alpha_g \leq 1$ such that for any vertex $v \in V$, if $g$ fixes ${v}$ and $g|\partial T_{v} \ne id$, then
   \begin{align} \frac{ \mu_{\partial T}\left( \{ {\bf w} \in \partial T_{v} \mid g \cdot {\bf w} \ne {\bf w}\} \right)}{\mu_{\partial T}(\partial T_{v})} \geq \alpha_g.\end{align}

The action $(\partial T, G, \mu_{\partial T})$ is \emph{locally non-degenerate} if and only if the action of every $g \in G$ on $\partial T$ is locally non-degenerate.

If an action of $g \in G$, or the action $(\partial T,G,\mu_{\partial T})$ is not locally non-degenerate, then it is \emph{locally degenerate}.
\end{defn}

In particular, the action of $g \in G$ is locally non-degenerate if, for any decreasing sequence of nested clopen sets $\{\partial T_{v_\ell}\}_{v_\ell \in V_\ell, \ell \geq 1}$, where $g|\partial T_{v_\ell} \ne id$, the ratio of the measure of the set of points in $\partial T_{v_\ell}$, not fixed by the action of $g$, to the measure of $\partial T_{v_\ell}$ has a lower bound. Intuitively, we want to avoid actions, where for some $\{\partial T_{v_\ell}\}_{v_\ell \in V_\ell, \ell \geq 1}$, the set of fixed points of $g$ occupies a larger and larger proper subset of $\partial T_{v_\ell}$, as $\ell \to \infty$. Another way to phrase that is by saying that the proportion of points in an open ball, moved by the action of $g \in G$, is at least $\alpha_g$ for any open ball on which $g$ acts non-trivially. 

As discussed in Section \ref{subsec-tree}, every minimal equicontinuous action $(X,G,\Phi)$ on a Cantor set $X$ has a representation as an action on the boundary $\partial T$ of a rooted spherically homogeneous tree $T$. More precisely, there exists a (non-unique) tree $T$ and a homeomorphism $\phi: X \to \partial T$, such that for any $g \in G$ the composition $\phi \circ g\circ \phi^{-1}$ is an automorphism of $T$.  The homeomorphism $\phi$ is determined by a choice of a descending sequence $\{U_\ell\}_{\ell \geq 0}$, $U_0 = X$, of clopen subsets of $X$ such that for each $U_\ell$ the translates $\cU_\ell = \{g \cdot U_\ell\}_{g \in G}$ form a finite partition of $X$. Under $\phi$, distinct sets in $\cU_\ell$ correspond to clopen sets $\{\partial T_v\}_{v \in V_\ell}$ in the tree. The group $G$ permutes the sets in the partition $\cU_\ell$, and so defines a permutation of the vertex set $V_\ell$, for $\ell \geq 1$. Inclusions of sets in $\cU_{\ell+1}$ into the sets in $\cU_\ell$ correspond to edges between vertices in $T$, and so $G$ acts on $T$ by automorphisms. It follows that there is an induced injective map $\phi_*: Homeo(X) \to Homeo(\partial T)$, and the pair of maps 
 \begin{align}\label{eq-reprphi}(\phi, \phi_*):(X, \Phi(G)) \to (\partial T,Homeo(\partial T)) \end{align}
gives a representation of $(X,G,\Phi)$ onto the boundary of the tree $T$. Any two such representations are conjugate to $(X,G,\Phi)$, and so to each other. The invariant ergodic measure $\mu$ on $X$ pushes forward to the counting measure $\mu_{\partial T}$ on $\partial T$. The representation $(\phi,\phi_*)$ is not unique, and should be viewed as a choice of coordinates for the action. Thus associated to a minimal equicontinuous action $(X,G,\Phi,\mu)$ there is an equivalence class $\left[(\partial T, \phi_*(\Phi(G)),\mu_{\partial T})\right]$ of actions on trees, where the equivalence relation is given by conjugacy of the induced actions on the boundary of the tree.

\begin{defn}\label{defn-noncaction}
A minimal equicontinuous action $(X,G,\Phi,\mu)$ is \emph{locally non-degenerate} if there exists a representation $(\phi, \phi_*)$ of the action as in \eqref{eq-reprphi}, such that the action of $\phi_*(\Phi(G))$ on $\partial T$ is locally non-degenerate with respect to the counting measure $\mu_{\partial T}$ on $\partial T$.
\end{defn}

Suppose $(X,G,\Phi,\mu)$ is a locally non-degenerate action with respect to a representation $(\phi,\phi_*)$. Suppose $(Y,G,\Psi,\nu)$ is another minimal equicontinuous actions, and $\psi: Y \to X$ conjugates $(X,G,\Phi,\mu)$ and $(Y,G,\Psi,\nu)$. Then $(\phi\circ \psi, (\phi \circ \psi)_*)$ is a representation for $(Y,G,\Psi,\nu)$ such that the action of $(\phi \circ \psi)_*(\Psi(G))$ on $\partial T$ is locally non-degenerate. Thus local non-degeneracy of an action is an invariant of topological conjugacy of actions.

Actions which are locally non-degenerate are common in the study of actions on Cantor sets. We show in Section \ref{subsec-automata}  that actions on $d$-ary trees generated by finite automata are locally non-degenerate. We expect that any action on a Cantor set which is conjugate to a holonomy action on a transverse section of a subset of a foliated $C^1$ manifold is locally non-degenerate. We give examples of actions which are locally degenerate in Section \ref{subsec-nonunonconst}. It is not clear at the moment what type of actions is prevalent among all possible minimal equicontinuous group actions on Cantor sets.

We now can state our main theorem, which specifies conditions under which $X$ contains a set of points with conjugate stabilizers which is both topologically and measure-theoretically large.

\begin{thm}\label{thm-mainmain}
Let $X$ be a Cantor set, let $G$ be a finitely generated group, and let $(X,G,\Phi,\mu)$ be a locally quasi-analytic minimal equicontinuous action. Then the following is true.
\begin{enumerate}
\item There exists a subgroup $H \subset G$ such that the set of points with stabilizers conjugate to $H$ is residual in $X$. 
\item In addition, suppose $(X,G,\Phi,\mu)$ is locally non-degenerate. Then the set of points with stabilizers conjugate to $H$ has full measure in $X$.
\end{enumerate}
\end{thm}

If $(X,G,\Phi,\mu)$ is essentially free, then both statements of Theorem \ref{thm-mainmain} hold with $H$ the trivial subgroup.

The proof of Theorem \ref{thm-mainmain} consists of two ingredients, which are of interest in their own right. Below we discuss these ingredients and their applications.

\medskip
To prove Theorem \ref{thm-mainmain}, we study the properties of points with trivial or non-trivial holonomy in $X$. Throughout the paper,  we denote the set of points with trivial holonomy by
  \begin{align}\label{eq-ptX0}X_0 = \{x\in X \mid G_x = [G]_x\}. \end{align}

\begin{thm}\label{thm-stabilizers}
Let $(X,G,\Phi)$ be a minimal equicontinuous action of a finitely generated group $G$ on a Cantor set $X$. Then the set $\{G_x \mid x \in X_0\}$ of stabilizers of points with trivial holonomy is a finite set of conjugate subgroups if and only if  the action $(X,G,\Phi)$ is locally quasi-analytic. \end{thm}

 The set of stabilizers of points with non-trivial holonomy may be infinite even for quasi-analytic actions, see for instance, Example \ref{eq-dihedral}.  The next step in the proof of Theorem \ref{thm-mainmain} is the sufficient condition under which the set of points with trivial holonomy has full measure in $X$.  

\begin{thm}\label{thm-holonomymeasure0}
Let $(X,G,\Phi,\mu)$ be a minimal equicontinuous action of a finitely generated group $G$ on a Cantor set $X$.  Suppose $(X,G,\Phi,\mu)$ is locally non-degenerate. Then the set $X_0$ of points with trivial holonomy has full measure with respect to $\mu$.
\end{thm}

The proof of Theorem \ref{thm-holonomymeasure0} uses the technique similar to that of Hurder and Katok \cite[Proposition 7.1]{HK1987}, who studied linear holonomy in $C^1$ foliations of smooth manifolds $M$. Our setting is very different from the setting of \cite{HK1987}, as for actions on Cantor sets there are no differentials and, in fact, the notion of linear holonomy does not make sense. The condition that the action $(X,G,\Phi,\mu)$ is locally non-degenerate imposes a degree of regularity which in the smooth setting comes from differentiability.

Bergeron and Gaboriau \cite{BG2004}, and also Ab\'ert and Elek \cite{AE2007}, gave examples of group actions on Cantor sets which are topologically free and not essentially free. By the discussion after Theorem \ref{thm-residualwoholonomy}, in these examples the set of points with non-trivial holonomy has full measure. We give more examples of this type in Section \ref{subsec-counterex}. As a corollary to Theorem \ref{thm-holonomymeasure0} we obtain a criterion when a topologically free minimal equicontinuous action is essentially free.

\begin{cor}\label{cor-topessentiallyfree}
If a minimal equicontinuous action $(X,G,\Phi,\mu)$ is locally non-degenerate and topologically free then it is essentially free.
\end{cor}

Kambites, Silva and Steinberg \cite{KSS2006} showed that an action on the boundary of a rooted $d$-ary tree of a group, generated by finite automata, is topologically free if and only if it is essentially free. We show  in Section \ref{subsec-unnonconstant} that actions generated by finite automata are locally non-degenerate, and recover the result of \cite{KSS2006} as a consequence of Theorem \ref{thm-holonomymeasure0}.

\bigskip
To finish, we outline some applications of our results, described in more detail in Section \ref{sec-applications}.

Let ${\rm Sub}(G)$ be the set of all closed subgroups of a finitely generated discrete group $G$ with the Chaubaty-Fell topology, see Section \ref{subsec-irs}. Then ${\rm Sub}(G)$ is compact and totally disconnected.

 For an action $(X,G,\Phi, \mu)$ consider the mapping
  \begin{align}\label{eq-stmapi} {\rm St}: X \to {\rm Sub}(G): x \mapsto G_x, \end{align}
which assigns to each $x \in X$ its stabilizer.  The group $G$ acts on ${\rm Sub}(G)$ by conjugation. Stabilizers of points in an orbit of $x \in X$ are conjugate, so the mapping ${\rm St}$ maps the orbit of $x \in X$ onto the orbit of $G_x$ in ${\rm Sub}(G)$ under  conjugation. The pushforward $\nu = {\rm St}_*\mu$ of the invariant ergodic measure $\mu$ is an ergodic invariant measure on ${\rm Sub}(G)$ \cite{AM1966}. This measure is said to be an    \emph{invariant random subgroup}, or simply IRS, following the terminology of  \cite{AGV2014,Bowen2014}. The mapping \eqref{eq-stmapi} need not be injective, and may have points of discontinuity. 

We discuss  in more detail in Section \ref{subsec-irs} some results about IRS for groups acting on the boundary of trees   available in the literature. We use now Theorems \ref{thm-stabilizers} and \ref{thm-holonomymeasure0} to give a partial answer to two open questions posed by Grigorchuk in \cite{Grig2011}. 

Namely, for the set $X_0$ of points without holonomy, defined by \eqref{eq-ptX0}, and called the set of $G$-typical points in \cite{Grig2011}, consider the topological closure
\begin{align}\label{eq-ZclsX_0}
	Z = \overline{\{G_x \mid x \in X_0\}}\subset{\rm Sub}(G).
\end{align}
On \cite[p.123]{Grig2011}, Grigorchuk asked when the IRS $\nu$ is supported on the set $Z$. Also,    \cite[Problem 8.2]{Grig2011} asks under what condition on the action the system $(X,G,\Phi,\mu)$ is measure-theoretically isomorphic to the system $(Z, G, \nu)$? We give partial answers to these questions below.

\begin{thm}\label{thm-application}
Let $X$ be a Cantor set, let $G$ be a finitely generated group, and let $(X,G,\Phi,\mu)$ be a minimal equicontinuous action. Suppose $(X,G,\Phi,\mu)$ is locally non-degenerate. Then the following holds.
\begin{enumerate}
\item The IRS $\nu = {\rm St}_*\mu$  is supported on $Z$. 
\item The action $(X,G,\Phi,\mu)$ is not LQA if and only if $\nu$ is non-atomic. 
\item If $(X,G,\Phi,\mu)$ is not LQA and the restriction ${\rm St}: X_0 \to {\rm Sub}(G)$ is injective, then ${\rm St}$ provides a measure-theoretical isomorphism between $(X, G,\Phi,\mu)$ and $(Z,G,\nu)$.
\end{enumerate}
\end{thm}

Theorem \ref{thm-application} describes the properties of the IRS $\nu$ induced on ${\rm Sub}(G)$ via the mapping \eqref{eq-stmapi} in the case when the set of points with trivial holonomy for the action $(X,G,\Phi,\mu)$ has full measure. If an action $(X,G,\Phi,\mu)$ is locally degenerate, then the set of points with non-trivial holonomy in $X$ may have positive measure, and the following problem is natural.

\begin{prob}\label{prob-4}
Let $(X,G,\Phi,\mu)$ be a minimal equicontinuous action on a Cantor set $X$, and suppose the set of points with non-trivial holonomy in $X$ has positive measure. Describe the properties of the IRS $\nu = {\rm St}_*\mu$ induced on ${\rm Sub}(G)$ via the map \eqref{eq-stmapi}.
\end{prob}

Related to the construction of the IRS above is the following construction. Define
\begin{align}\label{eq-wtildeX}
	\widetilde X=\overline{\{(x,G_x) \mid x\in X_0\}}\subset X\times Z
\end{align}
and denote by $\widetilde\Phi$ the product action of $G$ by $\Phi$ in the first coordinate and by conjugation in the second coordinate. Then the projection $\eta:\widetilde X\to X$ to the 
first coordinate satisfies $g \cdot \eta(x,G_x) = \eta(g \cdot x, g G_x g^{-1})$. By \cite[Proposition 1.2]{GW2014}, we know that $(\widetilde X,G,\widetilde\Phi)$ is minimal 
and the set  $\{ \widetilde{x}\in \widetilde{X} \mid  |\eta^{-1}(\eta(\widetilde{x}))|=1\}$ of singleton fibers is dense in $\widetilde{X}$. Thus the system $(\widetilde X,G,\widetilde \Phi)$ is an \emph{almost one-to-one} extension of $(X,G,\Phi)$. As a consequence of Theorem \ref{thm-stabilizers}, we show that for LQA actions such an extension is trivial in the following sense.

\begin{thm}\label{thm-alternative-LAQ}
Let $(X,G,\Phi)$ be a minimal equicontinuous action. 
Then $(X,G,\Phi)$ is locally quasi-analytic if and only if 
 the almost one-to-one extension $\eta:\widetilde X\to X$ is a conjugacy.
\end{thm}

\medskip

The rest of the paper is organized as follows. In Section \ref{subsec-equicts} we recall the necessary background on the properties of equicontinuous actions.  In Section \ref{subsec-tree} we discuss representations of minimal equicontinuous actions onto the boundaries of rooted spherically homogeneous trees. We discuss locally non-degenerate actions in Section \ref{subsec-unnonconstant}, giving examples which illustrate the definition. We also state sufficient conditions for an action on a rooted tree to be locally non-degenerate in terms of the geometry of the tree. Theorem \ref{thm-holonomymeasure0} which states that for a locally non-degenerate action the set of points with trivial holonomy has full measure and its corollaries are proved in Section \ref{sec-nontrivhol}. In Section \ref{subsec-counterex} we give an example of an action where the set of points with trivial holonomy has zero measure. Theorem \ref{thm-stabilizers} and the main Theorem \ref{thm-mainmain} are proved in Section \ref{sec-stabilizers}. In Section \ref{sec-applications} we discuss applications of our results proving Theorems \ref{thm-application} and \ref{thm-alternative-LAQ}.

{\bf Acknowledgements:} The authors thank Steve Hurder and Gabriel Fuhrmann for useful discussions. 


\section{Equicontinuous actions}\label{subsec-equicts}

 We recall basic properties of equicontinuous group actions on Cantor sets. Sections \ref{subsec-groupchainsconstr}-\ref{subsec-stabilizerschains} discuss topological properties of such actions, so we omit the measure from the notation in these sections.

An action $\Phi:G \to Homeo(X)$ of a finitely generated group $G$ on a Cantor set $X$ is \emph{equicontinuous} if $X$ admits a metric $d$ such that for all $\epsilon>0$ there is $\delta >0$ such that for any $g \in G$ and any $x,y \in X$ with $d(x,y) < \delta$ we have $d(g \cdot x, g \cdot y) < \epsilon$. An action $(X,G,\Phi)$ is \emph{minimal} if the \emph{orbit} $G(x)=\{g\cdot x \mid g \in G\}$ of every $x \in X$ is dense in $X$. In this paper, $(X,G,\Phi)$ is always minimal and equicontinuous.

\subsection{Group chains}\label{subsec-groupchainsconstr} Our main tool in working with group actions are \emph{group chains}. In this section we explain how a group chain defines a minimal equicontinuous action on a Cantor set. Actions on Cantor sets in terms of group chains were studied in \cite{DHL2016,FO2002,CP2008,CortezMedynets2016} and other works.

\begin{defn}\label{defn-groupchain}
Let $G$ be a finitely generated discrete group. A \emph{group chain} is a descending infinite sequence $\{G_\ell\}_{\ell \geq 0}: G=G_0 \supset G_1 \supset \cdots$ of finite index subgroups of $G$.
\end{defn}

For each $\ell \geq 0$ the coset space $X_\ell = G/G_\ell$ is finite, and there are inclusion maps 
  $$X_{\ell+1} \to X_\ell: gG_{\ell+1} \mapsto gG_\ell,$$
induced by the inclusion of groups $G_{\ell+1} \subset G_\ell$. The inverse limit space
  \begin{align}\label{eq-xinfty}X_\infty=\lim_{\longleftarrow}\{X_{\ell+1} \to X_\ell\} = \{ (g_0 G_0, g_1 G_1,\ldots) \mid g_{\ell+1}G_{\ell+1} \subset g_\ell G_\ell\} \subset \prod_{\ell \geq 0} X_\ell\end{align}
is a Cantor set in the relative topology, induced from the product (Tychonoff) topology on $\prod_{\ell \geq 0} X_\ell$. We denote by $(g_\ell G_\ell)_{\ell \geq 0}$ sequences of cosets in $X_\infty$. Basic sets in $X_\infty$ are given by 
  \begin{align}\label{eq-basicset}U_{g,m} = \{(h_\ell G_\ell)_{\ell \geq 0} \in X_\infty \mid h_m G_m = g G_m\}. \end{align}
The group $G$ acts on the left on every coset space $X_{\ell}$, permuting the cosets in $X_\ell$, and there is an induced action by homeomorphisms on the inverse limit
  \begin{align}\label{eq-gactionxinfty}G \times X_\infty \to X_\infty: (g, (g_0 G_0, g_1G_1,g_2G_2,\ldots)) \mapsto (gg_0 G_0, g g_1G_1, gg_2 G_2, \ldots).\end{align}
Since the action of $G$ on every coset space $X_\ell$, $\ell \geq 0$, is transitive, the action \eqref{eq-gactionxinfty} on $X_\infty$ is minimal. A standard way to define an ultrametric on the space $X_\infty$ is by
 \begin{align}\label{eq-metricd} D((g_\ell G_\ell)_{\ell \geq 0}, (h_\ell G_\ell)_{\ell \geq 0}) = \frac{1}{2^m}, \textrm{ where }m = \min\{\ell \geq 0 \mid g_\ell G_\ell \ne h_\ell G_\ell\}. \end{align}
This metric measures for how long sequences of cosets $(g_\ell G_\ell)_{\ell \geq 0}$ and $(h_\ell G_\ell)_{\ell \geq 0}$ coincide. Since $G$ acts on each coset space $X_\ell$, $\ell \geq 0$, by permutations, it acts on $X_\infty$ by isometries relative to the metric $D$, and so its action is equicontinuous.

  \subsection{Adapted clopen sets and group chain models}\label{subsec-chainmodels} In this section, we show how to associate a group chain to a minimal equicontinuous action $(X,G,\Phi)$, where $X$ is a Cantor set, and $G$ is a finitely generated discrete group.

Let $CO(X)$ denote the collection  of all clopen (closed and open) subsets of $X$, which form a basis for the topology of $X$. 
For $\phi \in Homeo(X)$ and    $U \in CO(X)$, the image $\phi(U) \in CO(X)$.  
The following   result is folklore, and a proof is given in \cite[Proposition~3.1]{HL2018}.
 \begin{prop}\label{prop-CO}
A minimal Cantor action $(X,G,\Phi)$ is  equicontinuous  if and only if, for the induced action $\Phi_* \colon G \times CO(X) \to CO(X)$, the orbit of every $U \in CO(X)$ is finite.
\end{prop}
 
\begin{defn}\label{defn-adapted}
Let $(X,G,\Phi)$ be a minimal equicontinuous action. A clopen set $U \in CO(X)$ is \emph{adapted} to the action $\Phi$ if $U$ is non-empty  and for any $g \in G$, 
if $\Phi(g)(U) \cap U \ne \emptyset$ then $\Phi(g)(U) = U$.
\end{defn}

That is, the translates of an adapted set by the action of $G$ form a partition of the Cantor set $X$.

\begin{defn}\label{defn-adaptednbhd}
Let $(X,G,\Phi)$ be a minimal equicontinuous action and $x \in X$. 
A properly descending chain of clopen sets $\cU_x = \{U_{\ell} \subset X  \mid \ell \geq 0\}$ is said to be an \emph{adapted neighborhood basis} at $x$  if
    $x \in U_{\ell +1} \subset U_{\ell}$ for all $ \ell \geq 0$ with     $\cap \,  U_{\ell} = \{x\}$, and  each $U_{\ell}$ is adapted to $\Phi$.
\end{defn}

The following result is folklore, a proof can be found in \cite[Appendix A]{DHL2016}.

\begin{prop}\label{prop-adpatedchain}
Let  $(X,G,\Phi)$   be a minimal equicontinuous    action. Given $x \in X$, there exists an adapted neighborhood basis $\cU_x$ at $x$ for the action $\Phi$.
 \end{prop}

For an adapted set $U \subset X$, consider the collection of elements in $G$ which preserve $U$, that is,
 \begin{equation}\label{eq-adapted}
G_U = \left\{g \in G \mid \Phi(g)(U) = U  \right\}. 
\end{equation}
Then $G_U$ is a subgroup of   $G$, called the \emph{stabilizer} of $U$, or the \emph{isotropy subgroup} of the action of $G$ at $U$.  Since the orbit of $U$ under the action of $G$ is finite, $G_U$ has finite index in $G$.

Fix $x \in X$, and let $\cU_x = \{U_\ell\}_{\ell  \geq 0}$ be an adapted neighborhood basis at $x \in X$. Denote by $G_\ell = G_{U_\ell}$ the stabilizer of $U_\ell$. Since $U_{\ell+1} \subset U_\ell$, every element in $G$ which stabilizers $U_{\ell+1}$ also stabilizers $U_\ell$, and so there are inclusions $G_{\ell+1} \subset G_\ell$. Thus associated to an adapted neighborhood basis $\cU_x$ at $x \in X$, there is an infinite group chain $\cG^x = \{G_\ell\}_{\ell \geq 0}$ of isotropy subgroups of $G$ at the sets of $\cU_x$.

For each $\ell \geq 0$, the translates $\{\Phi(g)(U_\ell)\}_{ g \in G}$ of $U_\ell \in \cU_x$ form a finite partition of $X$. Define a coding map $\kappa_\ell: X \to X_\ell$, where $X_\ell = G/G_\ell$, by
  $$\kappa_\ell(y) = gG_\ell \textrm{ if and only if } y \in \Phi(g)(U_\ell).$$
The maps $\kappa_\ell$ are equivariant with respect to the action of $G$ on $X$ and $X_\ell$. Taking the inverse limit with respect to the coset inclusions $X_{\ell+1} \to X_\ell$, we obtain a homeomorphism
  \begin{align}\label{eq-kappainfty}\kappa_\infty: X \to X_\infty=\lim_{\longleftarrow}\{X_{\ell+1} \to X_\ell\} \end{align}
which conjugates the actions of $G$ on $X$ and $X_\infty$, the latter being given by \eqref{eq-gactionxinfty}. By construction, we have $\kappa_\infty(x) = (eG_\ell)$, where $eG_\ell$ is the coset of the identity $e$ in $G/G_\ell$.

 For a given action $(X,G,\Phi)$ the choice of an adapted neighborhood system $\cU_x$ is not unique, and so the choice of a group chain $\cG^x$ associated to the action is not unique. The relations between the group chains representing the same conjugacy class of minimal equicontinuous actions were studied in detail in \cite{FO2002,DHL2016}, where the following result was proved.

\begin{prop}\cite{DHL2016}\label{prop-uniqueness}
Let $G$ be a finitely generated group, and let $\cG = \{G_\ell\}_{\ell \geq 0}$ and $\cG' = \{G_\ell' \}_{\ell \geq 0}$ be group chains with $G = G_0 = G_0'$. Let $(X_\infty, G)$ and $(X_\infty',G)$ be minimal equicontinuous actions associated to $\cG$ and $\cG'$ by \eqref{eq-xinfty} and \eqref{eq-gactionxinfty}. Then $(X_\infty, G)$ and $(X_\infty',G)$ are conjugate by a homeomorphism $\phi: X_\infty \to X_\infty'$ if and only if there exists a sequence of elements $\{g_\ell\}_{\ell \geq 0} \subset G$ such that for any $n \geq 0 $ and any $\ell > n$, $g_nG_n = g_\ell G_n$, and, possibly after passing to a subsequence, there are inclusions
  \begin{align}\label{eq-frm}G=G_0 \supset G_1' \supset g_1 G_1 g_1^{-1} \supset G_2' \supset g_2 G_2 g_2^{-1} \supset \cdots. \end{align}
In addition, $\phi$ is \emph{pointed}, that is, $\phi(eG_\ell) = (e G_\ell')$ if and only if one can choose $g_\ell = e$ in \eqref{eq-frm} for all $\ell \geq 0$.
\end{prop}

The condition  $g_nG_n = g_\ell G_n$ for $\ell > n$ in Proposition \ref{prop-uniqueness} ensures that the chain $\{g_\ell G_\ell g_\ell^{-1}\}_{\ell \geq 0}$ is a descending group chain as in Definition \ref{defn-groupchain}.

\begin{remark}\label{remark-divisors}
{\rm
The inclusions \eqref{eq-frm} impose restrictions on the indices of subgroups in the chains $\cG$ and $\cG'$. Indeed, if 
  $ G_\ell' \supset g_\ell G_\ell g_\ell^{-1} $ then the index $|G: G_\ell'|$ divides the index $|G: g_\ell G_{\ell}g_{\ell}^{-1}|$. In particular, if $|G: G_\ell'| = p_1 p_2 \cdots p_\ell$, where $p_\ell = |G_{\ell-1}: G_\ell|$, $ \ell \geq 1$, are distinct primes, and $|G: G_\ell| = d^\ell$ for some integer $d \geq 2$, then the actions defined by the group chains $\cG$ and $\cG'$ are not conjugate, since the set of primes which divide $d$ is finite.
}
\end{remark}

\subsection{Stabilizers of points and group chains} \label{subsec-stabilizerschains}
Let $(X,G,\Phi)$ be a minimal equicontinuous system, let $\cU = \{U_\ell\}_{\ell \geq 0}$ be an adapted neighborhood basis at $x \in X$, and let $\cG^x = \{G_\ell\}_{\ell \geq 0}$ be an associated group chain. The \emph{kernel} of the group chain $\cG^x$ is the subgroup
  $$\cK(\cG^x) = \bigcap_{\ell \geq 0} G_\ell$$
of elements in $G$ which fix $x$, so the kernel of the group chain at $x$ is the stabilizer $G_x = \cK(\cG^x)$.

If $y \in X$ is another point, then for every $\ell \geq 0$ there is $g_\ell \in G$ such that $y \in \Phi(g_\ell)(U_\ell)$. It follows that $\cG^y = \{g_\ell G_\ell g_\ell^{-1}\}_{\ell \geq 0}$ is a group chain at $y$, and we can compute the stabilizer at $y$
  $$G_y = \cK(\cG^y) = \bigcap_{\ell \geq 0}g_\ell G_\ell g_\ell^{-1}.$$
If $y$ is in the orbit of $x$, that is, $y  = h \cdot x$ for some $h \in G$, then we can choose $g_\ell = h$, and in this case the stabilizers $G_x$ and $G_y$ are conjugate subgroups of $G$. If $y$ is not in the orbit of $x$, then the stabilizers $G_x$ and $G_y$ need not be isomorphic, as the following example shows.

\begin{ex}\label{eq-dihedral}
{\rm
Let $G = \langle a,b \mid bab=a^{-1}, b^2 = e \rangle$ be the dihedral group, where $e$ denotes the identity in $G$, let $G_\ell = \langle a^{2^\ell}, b \rangle$ for $\ell >0$, and $G_0 = G$. This example is very well-studied. For instance, the action defined by this group chain is conjugate to the action of the iterated monodromy group associated to the quadratic Chebyshev polynomial, see \cite{Nekr}. Group chains representing this action were studied in \cite{FO2002} and  \cite[Example 7.5]{DHL2016}. 

Consider the dynamical system associated to the group chain $\{G_\ell\}_{\ell \geq 0}$ by the construction in Section \ref{subsec-groupchainsconstr}. One sees that for all $\ell \geq 1$ the cosets in $X_\ell = G/G_\ell$ are represented by the powers of $a$, and so for each $x \in X$, the orbit of $x$ is given by $G(x) = \{a^n \cdot x \mid n \in \mZ \}$. 

So let $x = (e G_\ell)_{\ell \geq 0} \in X_\infty$, then $G_x = \cK(\cG^x) = \langle b \rangle$. The isotropy groups of the points in the orbit $G(x)$ are conjugate to $G_x$, more precisely,
  $$G_{a^n \cdot x} = \langle a^n b a^{-n} \rangle.$$
If $y \in X_\infty$ is not in the orbit of $x$, then by \cite{FO2002,DHL2016} the stabilizer of $y$ is trivial, $G_y = \{e\}$. In particular, this means that the action of $G$ on $X_\infty$ defined by $\cG^x$ is topologically free.

We show that the stabilizers of the points in $G(x)$ are pairwise distinct subgroups of $G$. We argue by contradiction. Suppose that $G_x = G_{a^n \cdot x}$. Then
  $a^nba^{-n} = b$, which implies that 
   $$e = b^2 = ba^n b a^{-n} = b a b \, b a^{n-1} b \, a^{-n} = a^{-1} ba^{n-1} b a^{-n} = a^{-2n},$$
 which contradicts the fact that $G$ is an infinite group. So $G_x \ne G_{a^n \cdot x}$.
}
\end{ex}

\subsection{Counting measure on $X$} \label{subsec-measure}

Let $(X,G,\Phi)$ be a minimal equicontinuous action, then the closure $E = \overline{\Phi(G)} \subset Homeo(G)$ in the uniform topology is a profinite compact group, called the \emph{Ellis}, or \emph{enveloping group} \cite{Auslander1988,Ellis1969}. The group $E$ acts on $X$, the isotropy group $E_x = \{\widehat{g} \in E \mid \widehat{g}(x) = x\}$ of its action at $x$ is a closed subgroup of $E$, and we have $X = E/E_x$. An element $g \in G$ acts on $E$ via group multiplication by $\Phi(g)$. The Haar measure  $\widehat{\mu}$ on $E$ is invariant with respect to this action and ergodic. The measure $\widehat{\mu}$ on $E$ pushes down to a probability measure $\mu'$ on $X$, and with this measure $(X,G,\Phi,\mu')$ is uniquely ergodic \cite{CP2008}.

Given a group chain $\cG^x = \{G_\ell\}_{\ell \geq 0}$ associated to the action $(X,G,\Phi)$ at a point $x \in X$, one defines a counting measure $\mu$ on the space $X_\infty$ in \eqref{eq-xinfty} by setting for every basic set $U_{g,m}$, defined by \eqref{eq-basicset},
  \begin{align}\label{eq-mubern}\mu (U_{g,m}) = \frac{1}{|G: G_m|},\end{align}
where $|G: G_m|$ is the index of $G_m$ in $G$. This measure is easily seen to be invariant under the action of $G$. It is immediate that the pullback of $\mu$ to $X$ along the conjugating map \eqref{eq-kappainfty} coincides with $\mu'$. By a slight abuse of notation we denote $\mu'$ and $\mu$ by the same symbol $\mu$ in the rest of the paper.

\section{Actions on trees}\label{subsec-tree}

In this section, we represent a minimal equicontinuous action as an action on the boundary of a spherically homogeneous tree $T$. Actions on trees, especially self-similar actions on $d$-ary trees, $d \geq 2$, are an active topic in Geometric Group Theory, see \cite{Nekr,Grig2011,GNS2015} for surveys. 

A \emph{tree} $T$ consists of a set of vertices $V = \bigsqcup_{\ell \geq 0} V_\ell$, where $V_\ell$ is a finite vertex set at level $\ell$, and of edges joining vertices in $V_{\ell+1}$ and $V_\ell$, for all $\ell \geq 0$, such that every vertex in $V_{\ell+1}$ is joined by an edge to a single vertex in $V_\ell$. A tree is \emph{rooted} if $|V_0| = 1$.  A tree $T$ is \emph{spherically homogeneous} if there is a sequence $n=(n_1,n_2,\ldots)$, called the \emph{spherical index} of $T$, such that for every $\ell \geq 1$ a vertex in $V_{\ell-1}$ is joined by edges to precisely $n_\ell$ vertices in $V_\ell$. We assume that $n_\ell \geq 2$ for $\ell \geq 1$. 

A spherically homogeneous tree $T$ is $d$-ary if its spherical index $n=(n_1,n_2,\ldots)$ is \emph{constant}, that is, $n_\ell = d$ for some positive integer $d$. If $d = 2$, then $T$ is called a \emph{binary} tree. 

 The spherical index $n=(n_1,n_2,\ldots)$ of a tree $T$ is \emph{bounded}, if there is $M > 0$ such that $n_\ell \leq  M$ for $\ell \geq 1$. A spherical index $n$ is \emph{unbounded} if it is not bounded.

Let $(X,G,\Phi)$ be a minimal equicontinuous action, $\cU_x = \{U_\ell\}_{\ell \geq 0}$ be a neighborhood basis at $x \in X$, and let $\cG^x = \{G_\ell\}_{\ell \geq 0}$ be an associated group chain as in Section \ref{subsec-equicts}. For $\ell \geq 0$, identify the vertex set $V_\ell$ with the coset space $X_\ell = G/G_\ell$. Join $v_\ell \in V_\ell$ and $v_{\ell+1} \in V_{\ell+1}$ by an edge if and only if $v_{\ell+1} \subset v_\ell$ as cosets. The obtained tree is spherically homogeneous, with entries $n_\ell = |G_{\ell-1}: G_{\ell}|$ in the spherical index $n$, for $\ell \geq 1$.

An infinite path in $T$ consists of a sequence of vertices $(v_\ell)_{\ell \geq 0}$ such that $v_{\ell+1}$ and $v_\ell$ are joined by an edge, for $\ell \geq 0$. The boundary $\partial T$ of $T$ is the collection of all infinite paths in $T$, and so it is the subspace
  $$\partial T = \{(v_\ell)_{\ell \geq 0} \subset \prod_{\ell \geq 0} V_\ell \mid v_{\ell+1} \textrm{ and }v_\ell \textrm{ are joined by an edge}\}.$$
The space $\partial T$ is a Cantor set with the relative topology from the product topology on $ \prod_{\ell \geq 0} V_\ell $. It is immediate that the identification of vertex sets $V_\ell$ with coset spaces $X_\ell$ induces an identification of $\partial T$ with the inverse limit space $X_\infty$ defined by \eqref{eq-xinfty}, with points in $X_\infty$ corresponding to infinite paths in $\partial T$, and clopen sets $U_{g,m}$ defined by \eqref{eq-basicset} corresponding to cylinder sets
  \begin{align}\label{eq-cylinderset}\partial T_{v_m} = \{ (u_\ell)_{\ell \geq 0} \in \partial T \mid v_m = u_m\}, \end{align}
  where $v_m$ is the vertex in $V_m$ identified with the coset $gG_m$ in $X_m$. The counting measure \eqref{eq-mubern} pushes forward to the measure on $\partial T$ which we denote by $\mu_{\partial T}$.

The group $G$ acts on vertex levels $V_\ell = X_\ell$, $\ell \geq 0$ by permutations. Since the action of $G$ preserves the containment of cosets, the action preserves the connectedness of the tree $T$, that is, the vertices $v_{\ell} \in V_{\ell}$ and $v_{\ell+1} \in V_{\ell+1}$ are joined by an edge if and only if for any $g \in G$ the images $g \cdot v_\ell \in V_{\ell}$ and $g \cdot v_{\ell+1} \in V_{\ell+1}$ are joined by an edge. Thus every $g \in G$ defines an automorphism of the tree $T$, and we can consider $G$ as a subgroup of the group of tree automorphisms $Aut(T)$. 

The composition of the map \eqref{eq-kappainfty} with the identification $X_\infty \to \partial T$ is a homeomorphism $\phi: X \to \partial T$. The action of $G$ on vertex levels $V_\ell$, $\ell \geq 0$ induces an action of $G$ on $\partial T$ by left translations, defined by \eqref{eq-gactionxinfty}. Thus there is the induced map $\phi_*: \Phi(G) \to Homeo(\partial T)$, and $\Phi(G)$ is identified with a subgroup of $Homeo(\partial T)$. The action $(\partial T, \Phi(G))$ is minimal equicontinuous, with the unique ergodic invariant measure $\mu_{\partial T}$ defined by the pushforward of \eqref{eq-mubern}. A pair of maps 
  \begin{align}\label{eq-treerep}(\phi, \phi_*): (X,\Phi(G)) \to (\partial T, Homeo(\partial T))\end{align}
is called a \emph{tree representation} of $(X,G,\Phi)$. 

\begin{remark}\label{remark-nonuniquerepresentation}
{\rm
The choice of a group chain associated to an action $(X,G,\Phi)$ is not unique, and consequently the choice of a tree representation \eqref{eq-treerep} is not unique. By an argument similar to the one in Remark \ref{remark-divisors} one can obtain a necessary condition under which trees $T$ and $T'$ with respective  spherical indices $n = (n_1,n_2,\ldots)$ and $n' = (n_1',n_2',\ldots)$ admit conjugate actions of the same group. Namely, there must exist subsequences $\{{i_\ell}\}_{\ell \geq 0}$ and $\{{i_\ell}'\}_{\ell \geq 0}$ such that for all $\ell \geq 0$ the product $n_1 n_2 \cdots n_{i_\ell}$ divides the product $n_1'n_2' \cdots n_{i_\ell}'$, and $n_1'n_2' \cdots n_{i_\ell}'$ divides $n_1 n_2 \cdots n_{i_{\ell+1}}$.

For instance, let $n = (p_1,p_2,\ldots)$ where $\{p_1,p_2,\ldots\}$ are distinct primes, and let $n'$ be bounded, that is, there is a constant $m \geq 0$ such that for all $\ell \geq 1$ we have $n_\ell' \leq m$. Then  the set of divisors $\{k \in \mN \mid k|n_\ell' \textrm{ for some } \ell \geq 1\}$ is finite, and so there is no map between $T$ and $T'$ which preserves the tree structure.
}
\end{remark}
   
\section{Locally non-degenerate actions on trees}\label{subsec-unnonconstant}

In this section we discuss the notion of a locally non-degenerate action, which we introduced in Section \ref{sec-intro}. We give a criterion for an action on a tree to be locally non-degenerate in terms of the geometry of the tree. We show that actions on rooted $d$-ary trees  generated by finite automata are locally non-degenerate. We also give examples of actions on trees with bounded or unbounded spherical index which are locally degenerate.

We use the notation of Section \ref{subsec-tree}. In particular, a clopen subset $\partial T_v$ defined in \eqref{eq-cylinderset}, is a subset of the boundary $\partial T$ of the tree $T$ containing all infinite paths passing through the vertex $v \in V$. We start by restating Definition \ref{defn-uniformnonconstanti} below for the convenience of the reader.

\begin{defn}\label{defn-uniformnonconstant}
Let $T$ be a spherically homogeneous tree, and let $G \subset Aut(T)$. The action of $g \in G$ on $\partial T$ is \emph{locally non-degenerate} if there exists $0 < \alpha_g \leq 1$ such that for any vertex $v \in V$, if $g$ fixes ${v}$ and $g|\partial T_{v} \ne id$, then
   \begin{align} \frac{ \mu_{\partial T}\left( \{ {\bf w} \in \partial T_{v} \mid g \cdot {\bf w} \ne {\bf w}\} \right)}{\mu_{\partial T}(\partial T_{v})} \geq \alpha_g.\end{align}

The action $(\partial T, G, \mu_{\partial T})$ is \emph{locally non-degenerate} if and only if the action of every $g \in G$ is locally non-degenerate.

If an action of $g \in G$, or the action $(\partial T,G,\mu_{\partial T})$ is not locally non-degenerate, then it is \emph{locally degenerate}.
\end{defn}

It is useful to have a criterion for when a homeomorphism of a tree is locally non-degenerate in terms of the geometry of the tree $T$. This criterion is given in Proposition \ref{prop-indexandK} below. 

\begin{prop}\label{prop-indexandK}
Let $T$ be a spherically homogeneous tree with spherical index $n = (n_1,n_2,\ldots)$, and let $g \in Aut(T)$. Suppose the following two conditions are satisfied:
\begin{enumerate}
\item The spherical index of $T$ is bounded, that is, there is $M > 0$ such that $n_\ell \leq  M$ for $\ell \geq 1$.
\item There is an integer $K_g> 0$ such that for any $\ell \geq 0$ and any vertex $v_\ell \in V_\ell$, if $g$ fixes ${v_\ell}$ and $g|\partial T_{v_\ell} \ne id$, then there exists $m \geq \ell$ and a vertex ${w_m} \in V_m \cap T_{v_\ell}$ such that $g \cdot w_m \ne w_m$ and $m - \ell \leq K_g$. 
\end{enumerate}
Then the action of $g$ is locally non-degenerate with respect to the probability measure $\mu_{\partial T}$.
\end{prop}

\proof The map $g$ acts non-trivially on every point in the clopen set $\partial T_{w_m}$. Therefore, we have
$$\frac{\mu_{\partial T}(\{ {\bf u} \in \partial T_{v_\ell} \mid g \cdot {\bf u} \ne {\bf u}\})}{\mu_{\partial T}(\partial T_{v_\ell})} \geq \frac{\mu_{\partial T}( \partial T_{w_m} )}{\mu_{\partial T}(\partial T_{v_\ell})} = \frac{n_1n_2 \cdots n_\ell}{n_1n_2 \cdots n_\ell \cdots n_m} = \frac{1}{n_{\ell+1} \cdots n_{m}} \geq \frac{1}{n_{\ell+1} \cdots n_{\ell+K_g}}.$$
Since $n_\ell \leq  M$ for all $\ell \geq 1$, we obtain that 
 $$\frac{\mu_{\partial T}(\{ {\bf u} \in \partial T_{v_\ell} \mid g \cdot {\bf u} \ne {\bf u}\})}{\mu_{\partial T}(\partial T_{v_\ell})} \geq \frac{1}{M^{K_g}},$$
and $g$ is locally non-degenerate with $\alpha_g = 1/M^{K_g}$.
\endproof

Proposition \ref{prop-indexandK} gives sufficient conditions for an action to be locally non-degenerate. This conditions are not necessary, for instance, it is possible to have a locally non-degenerate action on a tree with strictly increasing spherical index.

\begin{ex}\label{odometer-ex}
{\rm
Let $G$ be a finitely generated group, acting freely on the boundary $\partial T$ of a spherically homogeneous tree $T$ with any spherical index. That is, for any $g \in G$ and any ${\bf u} \in \partial T$ if $g \cdot {\bf u} = {\bf u}$ then $g = id$, the identity element in $G$. Such an action is trivially locally non-degenerate, since no element has fixed points.
}
\end{ex}

\subsection{A class of locally non-degenerate actions}\label{subsec-automata}

We show that actions on $d$-ary trees of groups generated by finite automata are locally non-degenerate. Many well-known groups belong to this class, including the Basilica group and the Grigorchuk group, as well as iterated monodromy groups associated to quadratic post-critically finite polynomials,  see \cite{Nekr} for more specific examples and detailed discussions.

\subsubsection{Recursive definition of an automorphism of a $d$-ary tree}
Let $T$ be a $d$-ary tree, with the boundary $\partial T$ consisting of all infinite connected paths in $T$. As is common in the literature, we label the vertices in $T$ by finite words in the alphabet $\{0,1, \ldots, d-1\}$ as follows.  The root in $V_0$ is not labelled; $d$ vertices in $V_1$ are labelled by digits from $0$ to $d-1$. Every vertex in the set $V_\ell$ is labelled by a unique word of length $\ell$. If $v \in V_{\ell+1}$ and $w \in V_\ell$ are joined by an edge, and $v = s_1 s_2 \ldots s_{\ell+1}$, then $w = s_1 s_2 \ldots s_\ell$. It follows that every element of the boundary $\partial T$ can be uniquely represented by an infinite sequence $s_1s_2 \cdots$, where $s_\ell \in \{0,1,\ldots,d-1\}$, $\ell \geq 1$. More precisely, a path $s = s_1 s_2 \ldots$ passes through the vertex labelled by $s_1$ in $V_1$, by $s_1s_2$ in $V_2$ and, inductively, $s$ passes through the vertex labelled by $s_1s_2 \cdots s_\ell$ in $V_\ell$ for $\ell \geq 1$.

Let $w= w_1w_2 \cdots w_\ell$ be a word of length $\ell$. 
Using the word notation, a clopen subset $\partial T_w$ of $\partial T$, defined in \eqref{eq-cylinderset} is given by
  $$\partial T_w = \{s_1s_2\cdots s_m \cdots \in \partial T \mid s_i = w_i \textrm{ for }1 \leq i \leq \ell\}.$$
Since for all $\ell \geq 1$ the labels $w_\ell$ take values in the same set $\{0,1,\ldots,d-1\}$, for each $w \in V_\ell$ there is a homeomorphism 
  \begin{align}\label{eq-psiw}\psi_w: \partial T_w \to \partial T: w_1 w_2 \cdots w_\ell s_{\ell+1} s_{\ell+2} \cdots  \mapsto s_{\ell+1} s_{\ell+2} \cdots.\end{align}

Note that $\partial T = \partial T_0 \cup \partial T_1 \cup \cdots \cup \partial T_{d-1}$. We use the maps \eqref{eq-psiw} to define \emph{sections} as in \cite[Section 1.3.1]{Nekr}, and then to recursively define elements in $Aut(T)$ as in \cite[Section 1.4.2]{Nekr} (although we compose maps in the recursive formula in a different order than in \cite{Nekr}).

Let $h \in Aut(T)$, and let $\sigma_{h,1}$ be the non-trivial permutation of $V_1 = \{0,1, \ldots, d-1\}$ induced by the action of $h$ on $T$. For $k \in V_1$, the restriction of $h$ to $T_k$ gives the map $h: \partial T_k \to \partial T_{\sigma_{h,1}(k)}$. Using the identification \eqref{eq-psiw} we obtain a homeomorphism $h|_k: \partial T \to \partial T$ which we call a \emph{section} of $h$ at $k$. Thus $h|_k \in Aut(T)$ is the map uniquely defined by the concatenation of sequences
   $$h(k \, s_2 \, s_3 \ldots) = \sigma_{h,1}(k)\, h|_k(s_2 \, s_3 \ldots).$$

Formally, for $0 \leq k \leq d-1$, to define $h|_k \in Aut(T)$ we set for every infinite path $s \in \partial T$
 \begin{align}\label{eq-section} h|_k (s) =  \psi_{\sigma_{h,1}(k)} \circ h \circ \psi_k^{-1}(s). \end{align}
 Then we can write the element $h$ as the composition
 \begin{align}\label{eq-composition} h = (h|_{\sigma_{h,1}^{-1}(1)},\ldots,h|_{\sigma_{h,1}^{-1}(d-1)})\sigma_{h,1}, \end{align}
where for each $0 \leq k \leq d-1$ we write just $h|_{\sigma_{h,1}^{-1}(k)}$ instead of $\psi_k^{-1} \circ h|_{\sigma_{h,1}^{-1}(k)} \circ \psi_k$, suppressing the notation for $\psi_k$ for simplicity.

Intuitively,  \eqref{eq-composition} splits the action of $h$ on $s$ into two stages: first we apply the permutation $\sigma_{h,1}$ to $V_1$, and then a suitable automorphism to each subtree $T_k$, for $0 \leq k \leq d-1$.

\subsubsection{Actions generated by finite automata are locally non-degenerate} Suppose $T$ is a $d$-ary tree, and $G \subset Aut(T)$. Suppose the element $g \in Aut(T)$ can be computed by a finite automaton, which is equivalent to the set of sections $\cS_g=\{g|_v \mid v \in V_\ell, \, \ell \geq 0\}$ being finite \cite[Section 1.3]{Nekr}. 

Number the elements in $\cS_g$, that is, $\cS_g = \{h_1,\ldots,h_k\}$ for some $k \geq 1$. For each $1 \leq i \leq k$ such that $h_i \ne id$, there exists $\ell_i \geq 0$ and a vertex $w_{\ell_i} \in V_{\ell_i}$ such that $h_i(w_{\ell_i}) \ne w_{\ell_i}$. Then $h_i$ acts non-trivially on any path in the clopen set $\partial T_{w_{\ell_i}}$. Let 
\begin{align}\label{eq-alphai} \alpha_i = \mu_{\partial T} (\partial T_{w_{\ell_i}})/ \mu_{\partial T}(\partial T) = \mu_{\partial T} (\partial T_{w_{\ell_i}}), \end{align}
and
  \begin{align}\label{eq-alphaautomata}\alpha_g = \min \{\alpha_i \mid 1\leq i \leq k, \, h_i \ne id\}. \end{align}
Now, suppose $g \in G$ fixes a vertex $v \in V_\ell$. Then $g|\partial T_v = g|_v = h_i$ for some $1 \leq i \leq k$, and Definition \ref{defn-uniformnonconstant} is satisfied with constant $\alpha_g$ defined by \eqref{eq-alphaautomata}. Therefore, the action of $g \in G$ is locally non-degenerate.

\subsection{Actions which are locally degenerate}\label{subsec-nonunonconst} We give examples of actions on rooted trees which are locally degenerate. Example \ref{ex-notnonconst-unbounded} is that of an action on a spherically homogeneous tree with strictly increasing spherical index, that is, hypothesis (1) in Proposition \ref{prop-indexandK} does not hold for this example. Example \ref{ex-notnonconstant} is an example of an action on a $d$-ary tree for which hypothesis (2) in Proposition \ref{prop-indexandK} does not hold.

\begin{ex}\label{ex-notnonconst-unbounded}
{\rm
Let $T$ be a tree with spherical index $n = (n_1,n_2,\ldots)$, where $2 \leq n_1 < n_2 < \cdots$ is an increasing sequence of integers. Let $H \subset Aut(T)$ be a group which acts minimally and equicontinuously on $\partial T$. We define an element $c \in Aut(T)$ whose action on $\partial T$ is locally degenerate. Then the action of $G = \langle c,H \rangle$ on $\partial T$ is minimal, equicontinuous and locally degenerate. As before, we denote by $\mu_{\partial T}$ the counting measure on $\partial T$ defined in Section \ref{subsec-tree}.

To this end, let ${\bf v} = (v_\ell)_{\ell \geq 0}$ be a path in $T$. For $\ell \geq 0$, choose a vertex $w_{\ell+1} \in V_{\ell+1} \cap T_{v_\ell}$ such that $w_{\ell+1} \ne v_{\ell+1}$, and a vertex $z_{\ell+2} \in V_{\ell+2}$ such that $w_{\ell+1}$ and $z_{\ell+2}$ are joined by an edge. Let $c$ act non-trivially on every path in the clopen basic set $\partial T_{z_{\ell+2}}$, for $\ell \geq 1$, for example, by applying a cyclic permutation to the vertices in $V_{\ell+3}$ joined to $z_{\ell+2}$. Let $c$ act as the identity map outside of the set $\bigcup_{\ell \geq 1} \partial T_{z_{\ell+2}}$. Then for each clopen set $\partial T_{w_{\ell+1}}$, for $\ell \geq 1$, we have
  \begin{align*} \frac{\mu_{\partial T}(\{ {\bf u} \in \partial T_{w_{\ell+1}} \mid c \cdot {\bf u} \ne {\bf u}\})}{\mu_{\partial T}(\partial T_{w_{\ell+1}})} = \frac{\mu_{\partial T}(\partial T_{z_{\ell+2}})}{\mu_{\partial T}(\partial T_{w_{\ell+1}})} = \frac{1}{n_{\ell+2}}.\end{align*}
 Since the sequence $\{n_\ell\}_{\ell \geq 1}$ is increasing, the action of $c$ is locally degenerate.
}
\end{ex}

We now modify Example \ref{ex-notnonconst-unbounded} to build an action on a $d$-ary tree which is locally degenerate.

\begin{ex}\label{ex-notnonconstant}
{\rm 
 Let $T$ be a $d$-ary tree, for $d \geq 2$, and let $H \subset Aut(T)$ be a group which acts minimally and equicontinuously on $\partial T$. We define an element $c \in Aut(T)$ whose action on $\partial T$ is locally degenerate. Then the action of $G = \langle c,H \rangle$ on $\partial T$ is locally degenerate.

Let ${\bf v} = (v_\ell)_{\ell \geq 0}$ be a path in $T$. For $k \geq 1$, let $m_k = 2^k$. Let $w_{m_k}$ be a vertex in $V_{m_{k}}$ which is joined by an edge to $v_{m_k-1}$ and is distinct from $v_{m_k}$. Let $z_{m_{k+1}}$ be a vertex in $V_{m_{k+1}}$ joined by a path to $w_{m_k}$. This path consists of $2^k$ edges. Define $c$ so that it acts non-trivially on every path in the clopen set $\partial T_{z_{m_{k+1}}}$, and trivially outside of the set $\bigcup_{\ell \geq 1} \partial T_{z_{m_{\ell+1}}}$. Note that by construction for every $k \geq 1$ we have 
  $$\partial T_{w_{m_k}} \cap \left(\bigcup_{\ell \geq 1} \partial T_{z_{m_{\ell+1}}} \right)= \partial T_{z_{m_{k+1}}}.$$ 

Then for each clopen set $\partial T_{w_{m_k}}$, $k \geq 1$, we have
  \begin{align*} \frac{\mu_{\partial T}(\{ {\bf u} \in \partial T_{w_{m_k}} \mid c \cdot {\bf u} \ne {\bf u}\})}{\mu_{\partial T}(\partial T_{w_{m_k}})} = \frac{\mu_{\partial T}(\partial T_{z_{m_{k+1}}})}{\mu_{\partial T}(\partial T_{w_{m_k}})} =  {d^{2^{k}-2^{k+1}}} = d^{-2^{k}}.\end{align*}
Since $d^{-2^k} \to_{k \to \infty} 0$ then the action of $c$ on $\partial T$ is locally degenerate. 
}
\end{ex}

\section{Sets of points without holonomy of full measure}\label{sec-nontrivhol}

In this section we prove Theorem \ref{thm-holonomymeasure0}. This theorem gives necessary conditions for the set of points with trivial holonomy for the action $(X,G,\Phi,\mu)$ to have full measure. This theorem is one of the main ingredients of the proof of Theorem \ref{thm-mainmain}. 

\subsection{Lebesgue density} Let $(X,G,\Phi,\mu)$ be a minimal equicontinuous action. A choice of a tree representation for the action $(X,G,\Phi,\mu)$ gives $X$ an ultrametric $D$ \eqref{eq-metricd}. A Cantor set $X$ is a Polish space, and so the following result of Miller \cite{Miller2008} applies. 

Denote by $B(x,\epsilon) = \{ y \in X \mid D(x,y) < \epsilon\}$ an open ball around $x$ of radius $\e >0$.

\begin{thm} \cite[Proposition 2.10]{Miller2008} \label{thm-polishlebesgue}
Let $X$ be a Polish space, and suppose $X$ has an ultrametric $D$ compatible with its topology. Let $\mu$ be a probability measure on $X$, and let $A$ be a Borel set of positive measure. Then the \emph{Lebesgue density} of $x$ in $A$, given by
  \begin{align}\label{eq-lebesgue}\lim_{\epsilon \to 0} \frac{\mu(A \cap B(x,\epsilon))}{\mu(B(x,\epsilon))}\end{align}
exists and is equal to $1$ for $\mu$-almost every $x \in A$.
\end{thm}

\subsection{Proof of Theorem \ref{thm-holonomymeasure0}} For the convenience of the reader we first restate Theorem \ref{thm-holonomymeasure0} below.

\begin{thm}\label{thm-proveholonomy0}
Let $X$ be a Cantor set, let $G$ be a finitely generated group, and let $(X,G,\Phi,\mu)$ be a minimal equicontinuous action. Suppose $(X,G,\Phi,\mu)$ is locally non-degenerate. Then the set $X_0$ of points with trivial holonomy has full measure with respect to $\mu$.
\end{thm}

\proof Given $g \in G$, let ${\rm Fix}(g) = \{x \in X \mid g \cdot x = x\}$ be the set of fixed points of $g$. Note that ${\rm Fix}(g)$ is always a closed subset of $X$, and hence is Borel. The set ${\rm Fix}(g)$ may have positive measure, or measure zero. Assume that ${\rm Fix}(g)$ has positive measure.

Suppose $g \cdot x=x$. Recall from Definition \ref{defn-holpoints} that $g$ has trivial holonomy at $x \in X$ if $g$ fixes every point in an open neighborhood of $x$, and $g$ has non-trivial holonomy at $x$ otherwise.  We will show that under the hypothesis of the theorem the subset 
 $$\{x \in {\rm Fix}(g) \mid g \textrm{ has trivial holonomy at }x\in X\}$$
has positive measure in $X$, while the subset
 $$\{x \in {\rm Fix}(g) \mid g \textrm{ has non-trivial holonomy at }x\in X\}$$
has zero measure in $X$. 

If $g$ has trivial holonomy at $x \in {\rm Fix}(g)$, then it fixes every point in $B(x,\epsilon)$ for some $\epsilon >0$. Then the Lebesgue density of $x$ in ${\rm Fix}(g)$ given by \eqref{eq-lebesgue} exists and is equal to $1$. 

Next, suppose $g$ has non-trivial holonomy at $x \in {\rm Fix}(g)$. We will show that if the Lebesgue density of ${\rm Fix}(g)$ at $x$ exists, then it must be bounded away from $1$. 

Let $(\phi, \phi_*):(X, \Phi(G)) \to (\partial T,Homeo(\partial T))$ be a tree representation for the action, such that the action of $\Phi(G)$ on $\partial T$ is locally non-degenerate, and let $\phi(x) = (v_\ell)_{\ell \geq 0}$. For $\ell \geq 0$ let $U_\ell = \partial T_{v_\ell}$ be the clopen neighborhood of $(v_\ell)_{\ell \geq 0}$ consisting of all paths passing through the vertex $v_\ell$, and let
  $$W_\ell = \{ {\bf u} = (u_i)_{i \geq 0} \in \partial T_{v_\ell} \mid g \cdot {\bf u} \ne {\bf u}\}$$
be the subset of points in $U_\ell$ which are not fixed by the action of $g$. By the hypothesis the action of $g$ is locally non-degenerate, so there is a constant $0 < \alpha_g \leq 1$ such that
  $$\frac{\mu_{\partial T}(W_\ell)}{\mu_{\partial T}(U_\ell)} \geq \alpha_g.$$
Then we have
$$\frac{\mu_{\partial T}({\rm Fix}(g) \cap U_\ell)}{\mu_{\partial T}(U_\ell)} = \frac{\mu_{\partial T}(U_\ell) - \mu_{\partial T}(W_\ell)}{\mu_{\partial T}(U_\ell)} \leq 1 - \alpha_g.$$ 
It follows that if the Lebesgue density of ${\rm Fix}(g)$ at $\phi(x)$ exists, then it is bounded away from $1$. By Theorem \ref{thm-polishlebesgue} the subset of ${\rm Fix}(g)$ of points where the Lebesgue density does not exist or it exists but is not $1$ has measure zero in $\partial T$. Therefore, the subset of ${\rm Fix}(g)$ of points with non-trivial holonomy has measure zero.
Since the group $G$ is countable, the union
  $$\bigcup_{g \in G} \{x \in {\rm Fix}(g) \mid g \textrm{ has non-trivial holonomy at }x\in X\}$$
is a countable union of zero measure sets, and so has zero measure.
\endproof

Corollary \ref{cor-topessentiallyfree} is a reformulation of Theorem \ref{thm-holonomymeasure0} for the case of topologically free actions. Indeed, if an action $(X,G,\Phi, \mu)$ is topologically free, then $x \in X$ has non-trivial stabilizer if and only if it has non-trivial holonomy. By Theorem \ref{thm-holonomymeasure0} if $(X,G,\Phi,\mu)$ is topologically free and locally non-degenerate, then it is essentially free.

Kambites, Silva and Steinberg \cite{KSS2006} studied topologically and essentially free actions on $d$-ary rooted trees generated by finite automata. We recover their result as a consequence of Corollary \ref{cor-topessentiallyfree}.

\begin{cor}\label{cor-automatonaction}\cite[Theorem 4.3]{KSS2006}
Let $G$ be a group acting on a $d$-ary tree $T$, such that the action is transitive on each vertex level $V_\ell$, $\ell \geq 1$. Suppose the actions of $g \in G$ are computed by finite state automata. Then $(\partial T, G,\mu_{\partial T})$ is topologically free if and only if it is essentially free.
\end{cor}

\proof If a minimal action $(\partial T, G, \mu_{\partial T})$ is essentially free, then it contains a dense orbit of points with trivial stabilizer. Since the action is by homeomorphisms, for each $g \in G$ the set $K_g = \{x \in \partial T \mid g\cdot x \ne x \}$ is an open dense subset of $\partial T$. Then since $G$ is countable, the set $\bigcap_{g \in G} K_g$ of points with trivial stabilizer is a residual subset of $\partial T$. Thus $(\partial T, G, \mu_{\partial T})$ is topologically free.

For the converse, let $(\partial T, G, \mu_{\partial T})$ be topologically free.  We showed in Section \ref{subsec-automata} that a group action generated by finite automata is locally non-degenerate. Then by Corollary \ref{cor-topessentiallyfree} $(\partial T, G, \mu_{\partial T})$ is essentially free.
\endproof

\section{Sets of points with holonomy of full measure}\label{subsec-counterex}

Bergeron and Gaboriau \cite{BG2004} and Ab\'ert and Elek \cite{AE2007} gave examples of group actions on Cantor sets which are topologically free and not essentially free. By the discussion after Theorem \ref{thm-residualwoholonomy}, in these examples the set of points with non-trivial holonomy has full measure. We now describe another class of examples where the set of points with non-trivial holonomy has full measure, whose construction is somewhat easier than in \cite{AE2007,BG2004}.

Let $T$ be a tree with spherical index $n = (n_1, n_2,\ldots)$. As in Section \ref{subsec-automata} for $d$-ary trees, we label vertices in $V_\ell$, $\ell \geq 1$, by finite words of length $\ell$. More precisely, for $v \in V_\ell$ we have $v = w_1 w_2 \cdots w_\ell$, where $w_i \in \{0,1,\ldots, n_\ell - 1\}$. Then infinite paths in $\partial T$ are in bijective correspondence with infinite sequences $w_1w_2 \cdots$, where $w_\ell \in \{0,1,\ldots, n_\ell - 1\}$. Such a path passes through the vertices $v_\ell = w_1 \cdots w_\ell$, $\ell \geq 1$. Such a path ${\bf v}$ can be written either as $(v_\ell)_{\ell \geq 1}$, where $v_\ell = w_1 \cdots w_\ell$ is a vertex, or as an infinite sequence ${\bf v} = w_1w_2 \ldots$.

Define a homeomorphism $b$ of $\partial T$ as follows. Suppose that in the spherical index $n=(n_1,n_2, \ldots)$ we have $n_\ell \geq 3$ for $\ell \geq 1$. Recall that $T_{v_\ell}$ denotes the subtree of $T$ containing all paths through a given vertex $v_\ell = w_1 \cdots w_\ell$. Then the boundary $\partial T_{v_\ell}$ of $T_{v_\ell}$ consists of all infinite sequences starting with the finite word $v_\ell = w_1 \cdots w_\ell$.

The root $v_0$ is joined by edges to $n_1 \geq 3$ vertices in $V_1$, labelled by $0,1,\ldots, n_1-1$. Define $b$ to fix the vertices $0,1,\ldots, n_1-3$, and interchange the vertices $n_1-2$ and $n_1-1$. Define the action of $b$ on the rest of the tree by induction as follows.

Suppose the action of $b$ on $V_\ell$ is defined. Let $v_\ell \in V_\ell$, then $v_\ell$ is joined by edges to $n_{\ell+1}$ vertices in $V_{\ell+1}$ which are labelled by words of length $\ell+1$, namely $v_\ell 0, v_\ell 1, \ldots, v_\ell (n_{\ell+1}-1)$. If $b \cdot v_\ell \ne v_\ell$, then for $0 \leq k \leq n_{\ell+1}-1$ we set $b \cdot v_{\ell}k = (b \cdot v_\ell)k$, that is, the action of $b$ fixes the $(\ell+1)$-st entry in the sequence and only changes some of the preceding entries. If $b \cdot v_\ell = v_{\ell}$, then we define $b\cdot v_\ell k = v_\ell k$ for $0 \leq k \leq n_{\ell+1}-3$, and we set 
  \begin{align}\label{eq-jump}b \cdot v_\ell (n_{\ell+1} - 2) = v_\ell (n_{\ell+1}-1) \, \textrm{  and  } \, b\cdot v_\ell (n_{\ell+1} - 1) = v_\ell (n_{\ell+1}-2),\end{align}
  that is, $b$ interchanges $v_\ell(n_{\ell+1}-2)$ and $v_\ell(n_{\ell+1}-1)$, and fixes other vertices in $V_{\ell+1}$ joined to $v_\ell$.
  
The set ${\rm Fix}(b) \subset \partial T$ is non-empty. Indeed, let $s = s_1s_2 \cdots$ be a sequence. By definition of $b$  we have that $b \cdot s \ne s$ if and only if $s_\ell = n_\ell-2$ or $s_\ell = n_\ell-1$ for some $\ell \geq 1$. Therefore, every infinite sequence $s = s_1s_2 \cdots$ such that $s_\ell \leq n_\ell-3$ for all $\ell \geq 1$ is a fixed point of $b$. 

Also, $b$ is not the identity on any open set. Indeed, let $s = s_1s_2 \cdots$ be a fixed point of $b$, and let $U \subset \partial T$ be an open neighborhood of $s$. Then $U$ contains a basic clopen set $\partial T_{s_1 \cdots s_\ell}$, for some $\ell \geq 1$. Since $s$ is a fixed point of $b$, $s_{\ell+1} \ne n_{\ell +1} -2$ and $s_{\ell +1 } \ne n_{\ell+1}-1$. However, the open set $\partial T_{s_1 \cdots s_\ell}$ contains the union of clopen sets $\partial T_{s_1 \cdots s_\ell(n_{\ell+1}-2)} \cup \partial T_{s_1 \cdots s_\ell(n_{\ell+1}-1)}$ on which $b$ acts non-trivially by \eqref{eq-jump}. Therefore, $b$ is not the identity on $U$. Since $U$ is an arbitrary neighborhood of $s$, $s$ is a point with non-trivial holonomy. We have shown that $b$ has non-trivial holonomy at every point in ${\rm Fix}(b)$. 

\begin{thm}\label{thm-infiniteAssouad}
Let $T$ be a spherically homogeneous tree with spherical index $n = (n_1,n_2,\ldots)$ such that $n_{\ell+1}>2n_\ell$. Let $H$ be a group acting minimally and equicontinuously on $\partial T$, and let $G = \langle H,b\rangle \subset Aut(T)$ where $b$ is defined as above by \eqref{eq-jump}. Let $\mu_{\partial T}$ be the counting measure on $\partial T$. Then the following holds:
\begin{enumerate}
\item The action $(\partial T, G,\mu_{\partial T})$ is minimal and equicontinuous.
\item The set of points with non-trivial holonomy has full measure in $\partial T$.
\end{enumerate}
\end{thm}

\proof By assumption the orbits of points in $\partial T$ under the action of $H$ are dense in $\partial T$, so the action of $G = \langle H,b\rangle$ on $\partial T$ is minimal. Since $H$ and $b$ act by permutations on each level $V_\ell$, $\ell \geq 0$, the action of $G$ on $\partial T$ is equicontinuous. 

Let $\partial T$ be given an ultrametric \eqref{eq-metricd}. We show that the Lebesgue density of ${\rm Fix}(b)$ at every point in ${\rm Fix}(b)$ is $1$, and so ${\rm Fix}(b)$ must have positive measure. It follows that the set $\{{\bf v} \in \partial T \mid  [G]_{\bf v} \ne G_{\bf v}\}$ of points with non-trivial holonomy has full measure, since it is invariant under the action of $G$. 

Let ${\bf v} = (v_\ell)_{\ell \geq 0} \in {\rm Fix}(b)$, and let $U_\ell = \partial T_{v_\ell}$ be the subset of sequences in $\partial T$ which start with the finite word $v_\ell = w_1 \cdots w_\ell$ for $\ell \geq 1$. Set $U_0 = \partial T$ so $\mu_{\partial T}(U_0) = 1$. For $\ell \geq 1$ we have
 $$\mu_{\partial T}(U_\ell) = \frac{1}{n_1n_2 \ldots n_{\ell}}.$$ 
We first obtain an upper estimate on the measure of the complement $U_\ell - {\rm Fix}(b)$. 

For each $\ell \geq 0$ the element $b$ permutes two cylinder subsets of $U_\ell$, and we have for these subsets
  $$\mu_{\partial T}\left( \partial T_{v_\ell(n_{\ell+1}-2)} \cup \partial T_{v_\ell(n_{\ell+1}-1)} \right) = \frac{2}{n_1n_2 \ldots n_{\ell+1}}.$$
Next, each clopen set $\partial T_{v_\ell 0}, \partial T_{v_\ell1},\ldots, \partial T_{v_\ell(n_{\ell+1}-3)}$ contains two cylinder sets of measure $\frac{1}{n_1 n_2\cdots n_{\ell+2}}$ permuted by $b$. The measure of the union of these sets is equal to
  $$\frac{2(n_{\ell+1} - 2)}{n_1 n_2\cdots n_{\ell+2}} <  \frac{2 n_{\ell+1}}{n_1 \cdots n_{\ell+2}} = \frac{1}{n_1 n_2 \cdots n_{\ell}} \cdot \frac{2}{n_{\ell+2}}.$$

Inductively, we obtain that
  $$\mu_{\partial T} (U_\ell - {\rm Fix}(b)) = \frac{1}{n_1 \cdots n_{\ell}} \left( \frac{2}{n_{\ell+1}} + \frac{2(n_{\ell+1} - 2)}{n_{\ell+1} n_{\ell+2}} + \frac{2(n_{\ell+1} - 2)(n_{\ell+2} - 2)}{n_{\ell+1} n_{\ell+2}n_{\ell+3}} + \cdots \right) <  \frac{1}{n_1 \cdots n_{\ell}} \sum_{k\geq 1 } \frac{2}{n_{\ell+k}}.$$

By assumption $n_{\ell +k} > 2^{k-1} n_{\ell+1}$ for $k \geq 2$, therefore,
 $$\sum_{k \geq 1} \frac{2}{n_{\ell+k}} < \frac{2}{n_{\ell+1}} \left( 1 + \frac{1}{2} + \frac{1}{2^2} + \cdots \right) = \frac{4}{n_{\ell+1}},$$
and so 
  $$ \mu_{\partial T} (U_\ell - {\rm Fix}(b)) < \frac{4}{n_1n_2 \ldots n_{\ell+1}}.$$
It follows that 
  $$ \mu_{\partial T} ({\rm Fix}(b) \cap U_\ell)>  \frac{1}{n_1 \cdots n_{\ell}} - \frac{4}{n_1 \cdots n_{\ell+1}},$$
which implies that
$$   1 - \frac{4}{n_{\ell+1}}< \frac{\mu_{\partial T} ({\rm Fix}(b) \cap U_\ell)}{\mu_{\partial T}( U_\ell)} \leq 1.$$
Since $n_\ell \to \infty$ as $\ell \to \infty$, we obtain that the Lebesgue density of the set ${\rm Fix}(b)$ at $(v_\ell)_{\ell \geq 0}$ is $1$, and the statement of the theorem follows.
\endproof

\begin{remark}
{\rm
We note that the tree $T$ in Theorem \ref{thm-infiniteAssouad} has unbounded spherical index, so the generator $b$ does not satisfy the first condition in Proposition \ref{prop-indexandK}. Similarly, in the examples by Bergeron and Gaboriau \cite{BG2004} and Ab\'ert and Elek \cite{AE2007} the spherical index of the tree (equivalently, the index $|\Gamma_{\ell +1}: \Gamma_{\ell}|$ of subgroups in the group chain associated to the action) is not bounded.
}
\end{remark}

\section{Conjugate stabilizer subgroups}\label{sec-stabilizers}

In this section we prove Theorem \ref{thm-stabilizers} and the main Theorem \ref{thm-mainmain}.

\subsection{Proof of Theorem \ref{thm-stabilizers}} We restate the theorem in Theorem \ref{thm-stabilizers-proof} below in a slightly different form, which will be useful in the proof of Theorem \ref{thm-application}. 

Recall from the Introduction that a point $x \in X$ has trivial holonomy if for any $g \in G_x$, where $G_x$ is the stabilizer of the action at $x$, there exists an open set $U_g \owns x$ such that $g|_{U_g} = id$. We denote by $X_0$ the set of all points with trivial holonomy in $X$. 

\begin{thm}\label{thm-stabilizers-proof}
Let $(X,G,\Phi)$ be a minimal equicontinuous action of a finitely generated group $G$ on a Cantor set $X$. Let $x \in X_0$ be a point with trivial holonomy. Then the set  of stabilizers of the points in the orbit of $x$
  \begin{align}\label{eq-stabfinite} \{G_{g \cdot x} \mid g \in G\} = \{g \,G_x \,g^{-1} \mid g \in G\} \end{align}
is finite if and only if $(X,G,\Phi)$ is LQA. 

Moreover, if the action $(X,G,\Phi)$ is LQA, then for any other point with trivial holonomy $y \in X_0$ the stabilizer $G_y$ is conjugate to the subgroups in \eqref{eq-stabfinite}.
\end{thm}

\begin{remark}
{\rm
The restriction to points with trivial holonomy in Theorem \ref{thm-stabilizers-proof} is necessary. Indeed, if $x \in X$ is a point with non-trivial holonomy, then the set $\{G_{g \cdot x} \mid g \in G\}$ may be infinite even if $(X,G,\Phi)$ is topologically free (and so LQA), see Example \ref{eq-dihedral}. 
}
\end{remark}

\proof Let $x \in X_0$ be a point without holonomy, and let $\cU_x = \{U_\ell\}_{\ell \geq 0}$, $U_0=X$, be an adapted neighborhood system at $x$, see Section \ref{subsec-chainmodels} for terminology. Then $G_\ell = \{g \in G \mid g \cdot U_\ell = U_\ell\}$ is a group for any $\ell \geq 0$, with $G_0 = G$. Denote by $\Phi_\ell: G_\ell \to Homeo(U_\ell)$ or by $(U_\ell, G_\ell, \Phi_\ell)$ the induced action of $G_\ell$ on $U_\ell$.

Suppose $(X,G,\Phi)$ is LQA with $\epsilon \geq 0$, and let $\ell \geq 0$ be such that ${\rm diam}(U_\ell) < \epsilon$.  Then for any open set $W \subset U_\ell$, if $g|W = id|W$, then $g|U_\ell = id|U_\ell$. In particular, if $y \in U_\ell$ is without holonomy, then every element which fixes every point in an open neighborhood of $y$ in $U_\ell$ must fix every point in $U_\ell$. It follows that all points without holonomy in $U_\ell$ have equal stabilizers, that is, for all $y \in X_0 \cap U_\ell$ we have $G_y = \ker\{\Phi_\ell:G_\ell \to Homeo(U_\ell)\}$. The group $\ker (\Phi_\ell)$ is a normal subgroup in $G_\ell$, but it need not be normal in $G$. 

Let $g \notin G_\ell$, then $\widehat{x} = g \cdot x \in g\cdot U_\ell$, with $g \cdot U_\ell \cap U_\ell = \emptyset$. Moreover, $G_{\widehat{x}} = g G_x g^{-1}$, and $h \in G$ fixes an open neighborhood of $x$ if and only if $g h g^{-1}$ fixes an open neighborhood of $\widehat{x}$, so $\widehat{x}$ is a point without holonomy in $g \cdot U_\ell$.

Let $\widehat{y} \in g \cdot U_\ell \cap X_0$ be another point without holonomy. Then there exists a point $y \in U_\ell \cap X_0$  without holonomy such that $\widehat{y} = g \cdot y$. Then we have
  \begin{align} G_{\widehat{y}} & = g G_y g^{-1} =  g G_x g^{-1} = G_{\widehat{x}},\end{align}
so for all $\widehat{y} \in g \cdot U_\ell \cap X_0$ the stabilizers are equal, $G_{\widehat{y}} = g \ker (\Phi_\ell) g^{-1}$. Since the orbit of $U_\ell$ under the action of $G$ is finite, $\ker (\Phi_\ell)$ has a finite number of distinct conjugates in $G$, and the set $\{G_x \mid x \in X_0\}$ is finite. This proves that if $(X,G,\Phi)$ is LQA, then for any $x \in X_0$ the set $\{G_{g \cdot x} \mid g \in G\}$ is a finite set of conjugate subgroups, and the stabilizer of any point $y \in X_0$ without holonomy is conjugate to the subgroups in $\{G_{g \cdot x} \mid g \in G\}$.

We prove the converse by showing that if $(X,G,\Phi)$ is not LQA then the set $\{G_x \mid x \in X_0\}$ is infinite.

Suppose that the action $(X,G,\Phi)$ is not LQA. Given $x \in X_0$, we show that the set $\{G_{g \cdot x} \mid g \in G\}$ is infinite by induction. Namely, we will construct an increasing collection of finite subsets $Y_n$, $n \geq 1$, of the orbit $G(x) = \{z \in X_0 \mid z = g \cdot x, \, g\in G\}$, such that for $n \geq 1$ the cardinality of $Y_n$ is $2^n$, and all points $Y_n$ have pairwise distinct stabilizers. Points in $Y_n$ will be labelled by words $k_1 \cdots k_n$ of length $n$, where $k_i \in \{0,1\}$ for $1 \leq i \leq n$. The points will be chosen so that $y_0 = x$, and for $n \geq 1$ $y_{k_1 \cdots k_n 0} = y_{k_1 \cdots k_n}$.

We now start the construction of the subsets $Y_n$, $n \geq 1$. As in the first part of the proof, $\cU_x = \{U_\ell\}_{\ell \geq 0}$, $U_0=X$, is an adapted neighborhood system at $x$. We first construct $Y_1$.

Since $(X,G,\Phi)$ is not LQA, there exists an element $g_0 \in G$ which satisfies $g_0|U_1 = id$, and such that $g_0|(X - U_1)$ is not the identity. Choose $z \in X - U_1$ such that $g_0\cdot z \ne z$. By continuity there is an open neighborhood $O \owns z$ such that $g_0$ fixes no point in $O$. Choose an index $s_1 \geq 1$ large enough so that for some $h_{s_1} \in G$ we have $z \in h_{s_1} \cdot U_{s_1} \subset O$, then for any $z' \in h_{s_1} \cdot U_{s_1}$ we have $g_0\cdot z' \ne z'$. Set $W_0 = U_{s_1}$, and $W_1 = h_{s_1} \cdot U_{s_1}$. 

Then for any $z \in W_0$ we have $g_0 \in G_z$, in particular, for $z = x$. So we choose $y_0 = x$. For any $z \in W_1$ we have $g_0 \notin G_{z}$. Since the action is minimal and the set $W_1$ is open, we can choose $y_1 \in G(x) \cap W_1$. Since $x \in X_0$ and $y_1$ is in the orbit of $x$, then $y_1 \in X_0$. We set $Y_1 = \{y_0,y_1\}$.

Now suppose we are given a finite set of points labelled by words of length $n$, namely,
  $$Y_n = \{ y_{k_1 \cdots k_n} \mid k_i \in \{0,1\}, \, 1 \leq i \leq n\},$$ 
 and a finite collection of clopen sets, labelled by words of length $i$, for $1 \leq i \leq n$, namely,
   $$W_{k_1}, W_{k_1 k_2}, \cdots, W_{k_1k_2 \cdots k_n}, \textrm{ where }k_i \in  \{0,1\} \textrm{ for }1 \leq i \leq n.$$ 
We assume that these collections of sets have the following properties:
\begin{enumerate}
\item For every $1 \leq i < n$ and every word ${k_1\cdots k_{i+1}}$ we have an inclusion $W_{k_1\cdots k_{i+1}} \subset W_{k_1 \cdots k_i}$. That is, every set $W_{k_1 \cdots k_i}$ labelled by a word of length $i$ contains precisely two clopen sets labelled by words of length $i+1$, $W_{k_1 \cdots k_i0}$ and $W_{k_1 \cdots k_i1}$.
\item For every $1 \leq i < n$ and every word $k_1k_2 \cdots k_i$ there is an element $g_{k_1k_2 \cdots k_i0}\in G$ such that the restriction $g_{k_1k_2 \cdots k_i0}|W_{k_1k_2 \cdots k_i0}$ is the identity, and for every $z \in W_{k_1k_2 \cdots k_i1}$ we have $g_{k_1k_2 \cdots k_i0} \cdot z \ne z$. 

\item For every $1 \leq i \leq n$ and every $W_{k_1k_2 \ldots k_i}$ we have $y_{k_1k_2 \ldots k_i} \in W_{k_1k_2 \ldots k_i}$, where $y_{k_1k_2 \ldots k_i} \in Y_i \subset G(x)$. Also, $y_0 = x$, and for $1 \leq i < n$ we have $y_{k_1 \cdots k_{i} 0} = y_{k_1 \cdots k_{i}}$, so $Y_{i} \subset Y_{i+1}$. In particular, we have $y_{0 \cdots 0} = x$ for any word of zeros of length $i$.
\end{enumerate}

We claim that for any two distinct points $y_{j_1j_2 \ldots j_n}, y_{k_1k_2 \ldots k_n} \in Y_n$ the stabilizers $G_{y_{j_1j_2 \ldots j_n}}$ and $G_{y_{k_1k_2 \ldots k_n}}$ are distinct. Indeed, consider the words $j_1j_2 \ldots j_n$ and $k_1k_2\ldots k_n$, and let $s$ be the first digit such that $j_s \ne k_s$. Without loss of generality, we can assume that $j_s = 0$ and $k_s = 1$. Then by (1) we have that $y_{j_1j_2 \ldots j_n} \in W_{j_1j_2 \ldots j_n} \subset W_{j_1j_2 \ldots j_{s-1}0}$, and then by (2) we have $g_{j_1j_2 \ldots j_{s-1}0}  \in G_{y_{j_1j_2 \ldots j_n}}$. Similarly, by (1) we have $y_{k_1k_2 \ldots k_n} \in W_{k_1k_2 \ldots k_n} \subset W_{j_1j_2 \ldots j_{s-1}1}$, and then by (2) we have $g_{j_1j_2 \ldots j_{s-1}0}  \notin G_{y_{k_1k_2 \ldots k_n}}$. Then $G_{y_{j_1j_2 \ldots j_n}} \ne G_{y_{k_1k_2 \ldots k_n}}$.

We now implement the inductive step and construct $Y_{n+1}$, $W_{k_1 \cdots k_{n+1}}$ and $g_{k_1 \cdots k_n 0}$, where $k_i \in \{0,1\}$ for $1 \leq i \leq n$.

For each $k_1k_2 \ldots k_n$ we have $y_{k_1k_2 \ldots k_n} \in W_{k_1k_2 \ldots k_n}$, where $y_{k_1k_2 \ldots k_n} \in G(x)$. Since the action is not LQA, there exists a clopen neighborhood $W$ of $y_{k_1k_2 \ldots k_n} $ properly contained in $W_{k_1k_2 \ldots k_n} $, and an element $g_{k_1k_2 \ldots k_n 0}  \in G$ such that the restriction $g_{k_1k_2 \ldots k_n 0}|W$ is the identity, while the restriction $g_{k_1k_2 \ldots k_n 0}|(W_{k_1k_2 \ldots k_n }  - W )$ to the complement of $W$ in $W_{k_1k_2 \ldots k_n}$  is not the identity. Let $z \in W_{k_1k_2 \ldots k_n}  - W $ be so that $g_{k_1k_2 \ldots k_n 0}(z) \ne z$. By continuity there is a neighborhood $W' \owns z$ such that for every $z' \in W'$ we have $g_{k_1k_2 \ldots k_n 0}(z') \ne z'$.

Choose an index $s_{n+1} \geq 0$ and $h_0, h_1 \in G$ such that $y_{k_1k_2 \ldots k_n} \in h_0 \cdot U_{s_{n+1}} \subset W$ and $h_1 \cdot U_{s_{n+1}} \subset W'$ for an adapted neighborhood $U_{s_{n+1}} \in \cU_x$. If $k_1 \cdots k_n$ is a finite word of $0$'s, then we choose $h_0 = id$. Set $W_{k_1k_2 \ldots k_n 0}  = h_0 \cdot U_{s_{n+1}}$, and $W_{k_1k_2 \ldots k_n 1}  = h_1 \cdot U_{s_{n+1}}$. Then for every point $y' \in W_{k_1k_2 \ldots k_n 0}$ we have $g_{k_1k_2 \ldots k_n 0} \in  G_{y'}$, and for every $z' \in W_{k_1k_2 \ldots k_n 1}$ we have $g_{k_1k_2 \ldots k_n 0} \notin  G_{z'}$. We set $y_{k_1k_2\ldots k_n0} = y_{k_1k_2 \ldots k_n}$. Since the action is minimal, we can choose $z' \in G(y_0) \cap W_{k_1k_2 \ldots k_n 1}$, and then set $y_{k_1k_2 \ldots k_n1} = z'$. The collections of neighborhoods $W_{k_1}, W_{k_1 k_2}, \cdots, W_{k_1k_2 \ldots k_{n+1}}$, of elements $g_{0}, g_{k_10}, \ldots, g_{k_1 \cdots k_n0}$ and of points $Y_{n+1} = \{y_{k_1k_2 \cdots k_{n+1}} \mid k_i \in \{0,1\}, \, 1 \leq i \leq n \}$ we have constructed satisfy (1)-(3).

For $n \geq 1$, we have $\#Y_n = \# \{y_{k_1k_2\ldots k_n} \mid k_i \in \{0,1\}, 1 \leq i \leq n\} = 2^n$, and all points in this set are contained in the orbit $G(x) = G(y_0)$ and have distinct stabilizers. It follows $\{G_{g \cdot x} \mid g \in G\}$ is infinite. This finishes the proof of the theorem.
\endproof

\subsection{Proof of Theorem \ref{thm-mainmain}} We restate the theorem for the convenience of the reader below.

\begin{thm}\label{thm-mainmain-proof}
Let $X$ be a Cantor set, let $G$ be a finitely generated group, and let $(X,G,\Phi,\mu)$ be a locally quasi-analytic minimal equicontinuous action. Then the following is true.
\begin{enumerate}
\item There exists a subgroup $H \subset G$ such that the set of points with stabilizers conjugate to $H$ is residual in $X$. 
\item Suppose in addition  that $(X,G,\Phi,\mu)$ is locally non-degenerate. Then the set of points with stabilizers conjugate to $H$ has full measure in $X$.
\end{enumerate}
\end{thm}

\proof By Theorem \ref{thm-residualwoholonomy} the set $X_0$ of points with trivial holonomy is residual in $X$. Choose $x \in X_0$ and let $H=G_x$, where $G_x$ is the stabilizer of $x$. By Theorem \ref{thm-stabilizers}, if the action is LQA then the stabilizer of any other point in $X_0$ is conjugate to $H$. This proves the first statement. 

For the second statement, assume in addition that the action is locally non-degenerate. Then by Theorem \ref{thm-holonomymeasure0} the set $X_0$ has full measure, which completes the proof.
\endproof

\section{Applications}\label{sec-applications}

The goal of this section is to prove Theorems \ref{thm-application} and \ref{thm-alternative-LAQ}. We start by recalling some background on the invariant random subgroups, as needed for the proof of Theorem \ref{thm-application}.

\subsection{Preliminaries on invariant random subgroups}\label{subsec-irs}

We denote by ${\rm Sub}(G)$ the space of closed subgroups of a finitely generated group $G$. The space ${\rm Sub}(G)$ is equipped with the Chaubaty-Fell topology. Open sets in this topology are given by \cite{AM1966,AGV2014,Vor2012,BGN2015}
 \begin{align}\label{eq-chaubatyset}U_{A,B} = \{H \subset {\rm Sub}(G) \mid A \subset H, \, B \cap H = \emptyset \}, \end{align}
where $A$ and $B$ are finite sets. The space ${\rm Sub}(G)$ is a compact totally disconnected space, and $G$ acts on ${\rm Sub}(G)$ by conjugation. 

\begin{defn}\label{defn-irs}
An \emph{invariant random subgroup (IRS)} $\nu$ is a Borel probability measure on ${\rm Sub}(G)$ invariant under the action of $G$ on ${\rm Sub}(G)$ by conjugation.
\end{defn}

Let $(X,G,\Phi,\mu)$ be a minimal equicontinuous action, and consider the mapping
  \begin{align}\label{eq-stmap} {\rm St}: X \to {\rm Sub}(G): x \mapsto G_x \end{align}
which assigns to each $x \in X$ its stabilizer. Stabilizers of points in the same orbit in $(X,G,\Phi,\mu)$ are conjugate, so \eqref{eq-stmap} maps the orbit of $x$ in $X$ onto the orbit of $G_x$ in ${\rm Sub}(G)$. The map \eqref{eq-stmap} need not be injective. For instance, if $(X,G,\Phi,\mu)$ is a free action, which means that for any $x \in X$  we have $G_x = \{e\}$ where $e$ is the identity in $G$, then the image of \eqref{eq-stmap} is the trivial subgroup.

The properties of the mapping \eqref{eq-stmap} were studied in many works. We recall the following result. 

\begin{lemma}\cite[Lemma 5.4]{Vor2012}\label{lemma-vorobets}
Let $G$ act on a Hausdorff topological space $X$ by homeomorphisms. Then
\begin{enumerate}
\item The mapping \eqref{eq-stmap} is Borel measurable.
\item The mapping \eqref{eq-stmap} is continuous at $x \in X$ if and only if $[G]_x = G_x$, that is, $x$ is a point without holonomy.
\item If a sequence of points $\{x_\ell\}_{\ell \geq 0}$ converges to $x \in X$, and the sequence of stabilizers $\{G_{x_\ell}\}_{\ell \geq 0}$ converges to a closed subgroup $H \in {\rm Sub}(G)$, then $[G]_x \subset H \subset G_x$.
\end{enumerate}
\end{lemma}

For the action $(X,G,\Phi,\mu)$ the measure $\mu$ pushes forward along \eqref{eq-stmap} to the ergodic IRS $\nu = {\rm St}_*\mu$. For instance, if $(X,G,\Phi,\mu)$ is a free action then $\nu$ is an atomic measure supported on a single point in ${\rm Sub}(G)$. At the other extreme, if $X = \partial T$ is a boundary of a $d$-ary tree $T$, and $G$ is a weakly branch group, then the stabilizers of all points in $X$ are pairwise distinct \cite[Proposition 8]{BGN2015}, and $\nu$ is non-atomic.
    
 By Lemma \ref{lemma-vorobets} the set $X_0 = \{x \in X \mid G_x = [G]_x\}$ defined in \eqref{eq-ptX0} contains all points at which the mapping \eqref{eq-stmap} is continuous. Recall from \eqref{eq-ZclsX_0} that we denote $Z = \overline{\{ G_x  \mid x \in X_0\}}$.  
 
 While $X$ is a Cantor set, since the map \eqref{eq-stmap} may be discontinuous, the set $\{G_x \mid x \in X\}\subset {\rm Sub}(G)$ is only a Polish space and may contain isolated points. For example, if $G$ is the Grigorchuk group, then stabilizers of points with non-trivial holonomy are isolated points in ${\rm Sub}(G)$ \cite{Vor2012}. Thus the closed set $Z$ need not be a subset of $\{G_x \mid x \in X\}\subset {\rm Sub}(G)$ but of course it is a subset of $\overline{\{G_x \mid x \in X\}}\subset {\rm Sub}(G)$. More precisely, we have the following statement.
  
\begin{lemma}\label{lemma-gw}\cite[Proposition 1.2]{GW2014} 
If an action $(X,G,\Phi)$ is minimal, then $Z$ is the unique minimal subset in $\overline{\{G_x \mid x \in X\}}\subset {\rm Sub}(G)$ for the action of $G$ on ${\rm Sub}(G)$ by conjugation.
\end{lemma}

The space ${\rm Sub}(G)$ is metrizable. One way to define a metric on this space is by a pullback from a metric on the space of Schreier graphs of subgroups of $G$, see \cite[Section 3]{AGV2014} or \cite{Grig2011,Bowen2015}. 
 
Let $S$ be a finite symmetric generating set for $G$. Given a subgroup $H \subset G$, construct the \emph{Schreier graph} $\Gamma_{G/H}$ as follows: the cosets in $G/H$ are the vertices of $\Gamma_{G/H}$, and two vertices $hH$ and $gH$ are joined by an edge, directed from $hH$ to $gH$ and labeled by $s \in S$ if and only if $gH = shH$. Edges in the graph are assigned unit length, and $\Gamma_{G/H}$ has a length metric $D_{\Gamma_{G/H}}$, that is, the distance between two points in $\Gamma_{G/H}$ is the length of the shortest path between these points. The graph $\Gamma_{G/H}$ has a distinguished vertex which is the coset of the identity $eH$, so $\Gamma_{G/H}$ is a pointed metric space. Denote by $B_{\Gamma_{G/H}}(r)$ a metric ball of radius $r$ in $\Gamma_{G/H}$ centered at $eH$. Given $H_1,H_2 \in {\rm Sub}(G)$, the metric balls $B_{\Gamma_{G/H_1}}(r)$ and $B_{\Gamma_{G/H_2}}(r)$ are isomorphic if and only if there exists an isometry $f: B_{\Gamma_{G/H_1}}(r) \to B_{\Gamma_{G/H_2}}(r)$ which preserves the labelling of edges and such that $f(eH_1) = eH_2$. 

Denote the set of all Schreier graphs associated to closed subgroups of $G$ by
  $${\rm Sch}(G,S) = \{\Gamma_{G/H} \mid H \in {\rm Sub}(G)\}.$$
We define a metric on ${\rm Sch}(G,S)$ by setting
  \begin{align}\label{eq-metricSchreier}D_{\rm Sch}(\Gamma_{G/H_1}, \Gamma_{G/H_2}) = \frac{1}{2^k}, && k = \max\{ r \geq 0 \mid B_{\Gamma_{G/H_1}}(r) \textrm{ and } B_{\Gamma_{G/H_2}}(r) \textrm{ are isomorphic}\}.\end{align}
The metric space ${\rm Sch}(G,S) $ is compact. The action of $G$ on ${\rm Sch}(G,S) $ is defined by setting $g \cdot \Gamma_{G/H} = \Gamma_{G/gHg^{-1}}$, for any $g \in G$. The graphs $\Gamma_{G/H}$ and $\Gamma_{G/gHg^{-1}}$ are isomorphic as metric spaces, but not necessarily as pointed metric spaces.  We think about the action of $G$ on ${\rm Sch}(G,S) $ as moving the distinguished vertex in $ \Gamma_{G/H} $ from $eH$ to the coset of $gH$. It is immediate that the map
  \begin{align}\label{eq-subgroupsgraphs}{\rm Sub}(G) \to {\rm Sch}(G,S): H \mapsto \Gamma_{G/H} \end{align}
is a homeomorphism which commutes with the action of $G$ on ${\rm Sub}(G) $ and $ {\rm Sch}(G,S)$.

\begin{lemma}\label{lemma-expgraphs}
Let $Y \subset {\rm Sch}(G,S)$ be a minimal closed subset invariant under the action of $G$. Then $Y$ is finite if and only if the action of $G$ on $Y$ is equicontinuous. 
\end{lemma}

\proof A finite subset $Y$ is discrete, so if $G$ acts on a finite discrete set $Y$ minimally, then $Y$ consists of a single periodic orbit. A periodic action on a finite set is trivially equicontinuous.

Suppose the action of $G$ on $Y$ is equicontinuous. 
Then given $\e = 1/2^k$, $k >1$, there is $\delta >0$ such that if $D_{\rm Sch}(\Gamma_{G/H_1}, \Gamma_{G/H_2}) < \delta$ then $D_{\rm Sch}(g \cdot \Gamma_{G/H_1}, g\cdot \Gamma_{G/H_2})< 1/2^k$ for all $g \in G$. It follows that for any $g \in G$, there is an isomorphism between balls of radius $k$ centered at  $gH_1 \in \Gamma_{G/H_1}$ and $gH_2 \in \Gamma_{G/H_2}$. Since $\Gamma_{G/H_1}$ and $\Gamma_{G/H_2}$ are covered by the union of such balls, it follows that $ \Gamma_{G/H_1} =  \Gamma_{G/H_2}$, and $Y$ consists of isolated points. By assumption $Y$ is a closed subset of a compact space ${\rm Sch}(G,S)$, therefore it must be a finite set. 
\endproof

A metric on the space ${\rm Sub}(G)$ can be defined as a pullback of the metric \eqref{eq-metricSchreier} along the map \eqref{eq-subgroupsgraphs}. Such a metric depends on the choice of the generating set $S$. However, the conclusion of Lemma \ref{lemma-expgraphs} is true for any finite generating set, i.e. it is independent of the particular choice.

The study of the action of $G$ by conjugation on ${\rm Sub}(G)$, and of the invariant measures on this space appears first in the work of Moore \cite{AM1966}, see also Ramsay \cite{Ramsay1971}. The term invariant random subgroup (IRS) was introduced in Ab\'ert, Glasner and Virag \cite{AGV2014} and in Bowen \cite{Bowen2014}. The study of IRS's for locally compact Lie groups and their lattices started with the work of Stuck and Zimmer \cite{SZ94}, and has developed rapidly in the recent years, see Gelander \cite{Gelander2018,Gelander2018-2} for recent surveys. Another problem of interest is the classification of the IRS for various countable groups \cite{BGK2015}. For instance, Vershik \cite{Vershik2012} classified IRS's for the infinite finitary symmetric group. Bowen \cite{Bowen2015} showed that a free non-abelian group has many IRS's, and Bowen, Grigorchuk and Kravchenko \cite{BGK2015} obtained a similar statement for the lamplighter group. IRS's for universal groups of intermediate growth were studied by Benli, Grigorchuk and Nagnibeda \cite{BGN2015}. Bencs and T\'oth \cite{BT2018} considered IRS's for a subgroup $Aut(T)$ of a $d$-ary tree $T$ generated by finitary alternating automorphisms, with applications to actions of weakly branch groups.  Zheng \cite{Zheng2019} studied IRS's for the full group of minimal $\mZ^d$-actions on Cantor sets, and for branch groups. Thomas and Tucker-Drob in \cite{TTD2014} and \cite{TTD2018} classified IRS's for diagonal inductive limits of respectively finite symmetric groups and finite alternating groups. Dudko and Medynets \cite{DM2019} studied IRS's of full groups with associated Bratteli diagrams admitting finite number of ergodic measures.

In this paper, we consider the IRS's defined by actions of finitely generated groups on rooted spherically homogeneous trees. We allow the tree to have any spherical index. For the class of locally non-degenerate actions, we study the support of the IRS depending on whether the action is locally quasi-analytic.

\subsection{Proof of Theorem \ref{thm-application}} We obtain Theorem \ref{thm-application} as a direct consequence of Theorem \ref{thm-application-sec}.

Recall that two measure-preserving systems $(X,G,\Phi,\mu)$ and 
$(Y,G,\Psi,\nu)$ are \emph{isomorphic in the measure-theoretical sense} if
there are invariant sets of full measure $S_X \subset X$ and $S_Y \subset Y$,  and a bi-measurable bijection $\psi: S_X \to S_Y$ 
such that $\mu(\psi^{-1}(A))=\nu(A)$ for all measurable $A\subset S_Y$, and 
$\psi$ conjugates the action of $G$ on $S_X$ and $S_Y$, that is,
$$\psi(\Phi(g) \cdot x) = \Psi(g)\cdot \psi(x) \textrm{ for all }x \in S_X.$$
In this case we call $\psi$ a \emph{measure-theoretical isomorphism}.

Recall that $X_0$ denotes the subset of $X$ consisting of points without holonomy, and $Z = \overline{\{ G_x  \mid x \in X_0\}}$. 

\begin{thm}\label{thm-application-sec}
Let $X$ be a Cantor set, let $G$ be a finitely generated group, and let $(X,G,\Phi,\mu)$ be a minimal equicontinuous action. Suppose the set $X_0$ of points without holonomy has full measure in $X$. Then the following holds:
\begin{enumerate}
\item The IRS $\nu = {\rm St}_*\mu$  is supported on $Z$. 
\item The action $(X,G,\Phi,\mu)$ is not LQA if and only if $\nu$ is non-atomic. 
\item If $(X,G,\Phi,\mu)$ is not LQA and the restriction ${\rm St}: X_0 \to {\rm Sub}(G)$ is injective, then ${\rm St}$ provides a measure-theoretical isomorphism between $(X, G,\Phi,\mu)$ and $(Z,G,\nu)$.
\end{enumerate}
\end{thm}

\proof The first statement is a direct consequence of the assumption that the set $X_0$ of points without holonomy has full measure in $X$.

If the action $(X,G,\Phi,\mu)$ is locally quasi-analytic (LQA), then by Theorem \ref{thm-stabilizers} $Z$ is finite. Since $Z$ is the support of $\nu$, then in this case $\nu$ must be atomic. So if $\nu$ is non-atomic, then $(X,G,\Phi,\mu)$ is not LQA. If $(X,G,\Phi,\mu)$ is not LQA, then by Theorem \ref{thm-stabilizers} for every point without holonomy $x \in X_0$ the set $O_{G_x} = \{G_{g\cdot x} \mid  g\in G\} = \{g G_x g^{-1} \mid g \in G\}$ is infinite. This set is the orbit of $G_x$ in ${\rm Sub}(G)$ under the action of $G$ by conjugation. 
By Lemma \ref{lemma-gw} $Z$ is minimal, so $O_{G_x}$ is dense in $Z$. Since $O_{G_x}$ is infinite, the closure $Z$ of $O_{G_x}$ is strictly larger than $O_{G_x}$, containing limit points of $O_{G_x}$. By Lemma \ref{lemma-gw} the orbit of every limit point is dense in $Z$, and it follows that $Z$ is perfect. Then $Z$ is a Cantor set. Since the measure $\nu$ supported on $Z$ is finite, it is non-atomic. This proves statement $(2)$.

In the case when $(X,G,\Phi,\mu)$ is not LQA under the additional assumption that the restriction ${\rm St}|_{X_0}$ is injective, the third statement follows from the fact that $\nu$ is a push-forward measure.
\endproof

\proof \emph{(of Theorem \ref{thm-application})} Since $(X,G,\Phi,\mu)$ is locally non-degenerate, by Theorem  \ref{thm-holonomymeasure0} the set $X_0$ of points without holonomy has full measure in $X$. Then the statement follows by Theorem \ref{thm-application-sec}.
\endproof

\subsection{Almost one-to-one extensions}\label{sec-meanequicont}

Recall that in \eqref{eq-wtildeX} we defined the set 
  $$\widetilde X=\overline{\{(x,G_x) \mid x\in X_0\}}\subset X\times Z,$$ 
and the factor  map $\eta:\widetilde X\to X$ is almost one-to-one.
Moreover, recall that any continuous action has a unique (up to conjugacy)
\emph{maximal equicontinuous factor}, in the sense that any other
equicontinuous factor is also a factor of this maximal one, see for 
instance \cite{Auslander1988}.

\begin{prop}[{\cite[V(6.1)5, page 480]{deVries1993}}]\label{prop-almost1-1MEF}
	If $(Y,G,\Psi)$ is a minimal extension of an equicontinuous system $
	(X,G,\Phi)$ via a corresponding factor map which is almost one-to-one,
	then $(X,G,\Phi)$ is the maximal equicontinuous factor of $(Y,G,\Psi)$. 
\end{prop}

We prove Theorem \ref{thm-alternative-LAQ} which provides the following
alternative characterization of LQA actions.

\begin{thm}
Let $X$ be a Cantor set, and let $G$ be a finitely generated group. Let $(X,G,\Phi)$ be a minimal equicontinuous action. 
Then $(X,G,\Phi)$ is locally quasi-analytic if and only if $\eta:\widetilde X\to X$ is a conjugacy.
\end{thm}
\begin{proof}
	First, assume that $(X,G,\Phi)$ is LQA.
	By Theorem \ref{thm-stabilizers}, $Z$ is finite, and the action of $G$
	on $Z$ by conjugation is periodic, and so equicontinuous.
	Then the product action of $G$ on $X\times Z$ is
	also equicontinuous, see \cite[Lemma 4 in Chapter 2]{Auslander1988}.
	This in turn implies that the orbit closure of every point in $X\times Z$ is minimal,
	see \cite[Lemma 3 in Chapter 2]{Auslander1988}.
	Now, using that $\widetilde X$ is the unique minimal subset in 
	$\overline{\{(x,G_x) \, \mid \, x \in X\}}\subset X\times{\rm Sub}(G)$
	by \cite[Proposition 1.2 (3)]{GW2014}, we get that $\widetilde X=X\times Z$.
	Finally, since $\eta:\widetilde X\to X$ is almost one-to-one, we have that
	$(X,G,\Phi)$ is the maximal equicontinuous factor of $(\widetilde X,G,\tilde\Phi)$
	by Proposition \ref{prop-almost1-1MEF}, but this immediately implies that $\eta$
	must be a conjugacy.  
	
	For the opposite direction, assume that $\eta$ is a conjugacy.
	Then the action of $G$ on $\widetilde X$ is equicontinuous.
	Moreover, since $Z$ is a factor system of $\widetilde X$ (simply by projecting
	in the second coordinate), we get that the action of $G$ on $Z$ by conjugation
	is equicontinuous as well, see \cite[Corollary 6 in Chapter 2]{Auslander1988}. 
	Hence by Lemma \ref{lemma-expgraphs}, $Z$ is finite and so, by Theorem \ref{thm-stabilizers}, the action $(X,G,\Phi)$ is LQA.
\end{proof}


\end{document}